\documentclass[12pt]{amsart}

\usepackage{amsmath,amssymb}
\usepackage{fullpage,subfig}

\usepackage{graphicx}
\usepackage{mathpazo}
\usepackage{amssymb}
\usepackage{epstopdf}
\usepackage{tikz}
\usepackage[all,cmtip]{xy}
\usepackage{color}
\usepackage[parfill]{parskip}

\theoremstyle{plain}
\newtheorem{thm}{Theorem}
\newtheorem{cor}[thm]{Corollary}
\newtheorem{lem}[thm]{Lemma}
\newtheorem{conj}[thm]{Conjecture}
\newtheorem{prop}[thm]{Proposition}

\theoremstyle{definition}
\newtheorem{defn}[thm]{Definition}
\newtheorem{rem}[thm]{Remark}

\newcommand{\tree}[3]{\text{ {$ {\text{\normalsize $ #1$}}-\!\!\!<^{#2}_{#3}$}}}

\renewcommand{\int}{\operatorname{int}}

\newcommand{\Ker}{\operatorname{Ker}}

\newcommand{\id}{\operatorname{id}}

\newcommand{\Artin}{\operatorname{Artin}}

\newcommand{\im}{\operatorname{Im}}
\newcommand{\Hom}{\operatorname{Hom}}

\newcommand{\Z}{\mathbb{Z}}
\newcommand{\J}{\mathbb{J}}
\newcommand{\Y}{\mathbb{Y}}
\newcommand{\z}{{\Z_2}}

\newcommand{\R}{\mathbb{R}}
\newcommand{\LL}{\mathbb{L}}

\newcommand{\W}{\mathbb{W}}
\newcommand{\SW}{\mathbb{SW}}
\newcommand{\SL}{\mathbb{SL}}
\newcommand{\SY}{\mathbb{SY}}
\newcommand{\SJ}{\mathbb{SJ}}
\newcommand{\HC}{\mathbb{HC}}

\newcommand\sL{{\sf L}}
\newcommand\sD{{\sf D}}
\newcommand\sY{{\sf Y}}
\newcommand\sJ{{\sf J}}
\newcommand{\sW}{{\sf W}}
\newcommand{\sSJ}{{\sf SJ}}
\newcommand{\sSY}{{\sf SY}}
\newcommand{\sSW}{{\sf SW}}

\newcommand{\cT}{\mathcal{T}}
\newcommand{\cW}{\mathcal{W}}

\newcommand{\cF}{\mathcal{F}}
\newcommand{\cG}{\mathcal{G}}

\newcommand{\cA}{\mathcal{A}}

\newcommand{\wT}{\widetilde \cT}

\newcommand{\sra}{\twoheadrightarrow}

\newcommand{\ra}{\longrightarrow}

\newcommand{\iinfty}{{\mathchoice
{\begin{minipage}{.15in}\includegraphics[width=.15in]{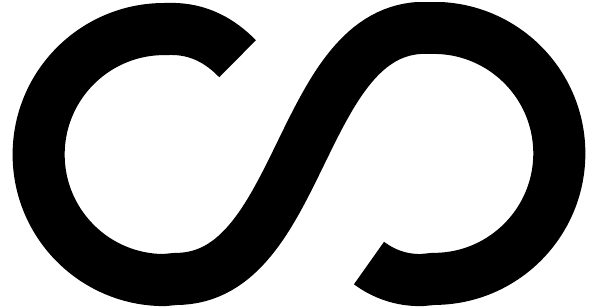}\end{minipage}}
{\begin{minipage}{.13in}\includegraphics[width=.13in]{infty2.pdf}\end{minipage}}
{\begin{minipage}{.11in}\includegraphics[width=.11in]{infty2.pdf}\end{minipage}}
{\begin{minipage}{.08in}\includegraphics[width=.08in]{infty2.pdf}\end{minipage}}
}}

\begin{document}

\title{Geometric filtrations of string links and homology cylinders}

\begin{abstract}
We show that the Artin representation on concordance classes of string links induces a well-defined epimorphism modulo order $n$ twisted Whitney tower concordance, and that the kernel of this map is generated by band sums of iterated Bing-doubles of any string knot with nonzero Arf invariant. We also continue J. Levine's work \cite{L1,L2,L3} comparing two filtrations of the group of homology cobordism classes of $3$-dimensional homology cylinders, one defined in terms of an Artin-type representation (the Johnson filtration) and one defined using clasper surgery (the Goussarov-Habiro filtration). In particular, the associated graded groups are completely classified up to an unknown $2$-torsion summand for the Goussarov-Habiro filtration, for which we obtain an upper bound, in a precisely analogous fashion to the classification of the Whitney tower filtration of link concordance. 

%Indeed, we show that Habegger's correspondence between the set of string links in a homology ball and the set of homology cylinders induces an isomorphism on the associated graded groups for the Goussarov-Habiro filtration for all $n$ not congruent to $1$ modulo $4$. In the remaining case of $n=4k+1$, the kernel is a finitely generated $2$-torsion group.

%A conjecture relating the target groups for the Whitney tower intersection theory to chord diagrams is formulated, and a map from string links to homology cylinders is used to exhibit the relationship between the Whitney tower and Habiro-filtrations.
\end{abstract}

\author[J. Conant]{James Conant}
\email{jconant@math.utk.edu}
\address{Department of Mathematics, University of Tennessee, Knoxville, TN 37996}

\author[R. Schneiderman]{Rob Schneiderman}
\email{robert.schneiderman@lehman.cuny.edu}
\address{Department of Mathematics and Computer Science, Lehman College, City University of New York, Bronx, NY 10468 }

\author[P. Teichner]{Peter Teichner} 
\email{teichner@mac.com} 
\address{Max-Planck Institut f\"ur Mathematik, Bonn, Germany and} 
\address{Department of Mathematics, University of California in Berkeley} 

%\shortauthors{J. Conant, R. Schneiderman and P. Teichner}


\maketitle
\section{Introduction}

This paper applies the techniques of our recent work on Whitney tower filtrations of classical link concordance \cite{CST0,CST1,CST3} to the settings of string links and $3$-dimensional homology cylinders. 
The problem is to understand certain geometric filtrations of links, string links and homology cylinders, respectively. We start by briefly summarizing the relevant filtrations; precise definitions are in subsequent sections. 
\begin{itemize}
\item[$\LL$\ :] The (framed) Whitney tower filtration $\LL \supset\W_0 \supset \W_1 \supset \W_2 \supset \cdots$ and twisted Whitney tower filtration $\W^\iinfty_n\supset\W_n$ on the set $\LL=\LL(m)$ of concordance classes of framed $m$-component links in the 3--sphere. Here $\W_n$ (respectively $\W^\iinfty_n$) consists of those links that bound a framed (respectively twisted) Whitney tower of order $n$ in the 4--ball \cite{CST1}.
\item[$\SL$\ :] 
\begin{itemize}
\item The analogous Whitney tower filtrations $\SW_n$ and $\SW^\iinfty_n$ on the group $\SL=\SL(m)$ of concordance classes of framed $m$-component string links (obtained from the usual closure operation from string links to links). 
\item The Johnson filtration $\SJ_n$ on $\SL$, defined as kernels of nilpotent Artin representations $\operatorname{Artin}_n\colon\mathbb{SL}\to \operatorname{Aut}_0(F/F_{n+2})$, where $F=F(m)$ is a free group on $m$ generators, $F_n$ are the terms in its lower central series and $\operatorname{Aut}_0(F/F_n)$ consists of those automorphisms of $F/F_n$ which are defined by conjugating each generator and which fix the product of generators.
\item The Goussarov-Habiro $Y$-filtration $\SY_n$ on $\SL$ consists of string links obtained from the unlink via surgeries along claspers with $n$ nodes.
\end{itemize}
\item[$\HC$\ :] The Johnson and Goussarov-Habiro filtrations generalize to filtrations $\J_n$ and $\Y_n$ respectively on the group $\HC=\HC(g,b)$ of homology cobordism classes of homology cylinders over a surface $\Sigma_{g,b}$ of genus $g$ with $b$ boundary circles. 
\end{itemize}

Various graded groups associated to these filtrations will also be defined precisely below, and will be denoted by the sans-serif versions of the above letters, for example $\sSW_n := \SW_n/\SW_{n+1}$. Our results are as follows:

\begin{thm} \label{thm:nilpotent}
The sets $\mathbb{SW}_n$ and $\mathbb{SW}_n^\iinfty$ are normal subgroups of $\mathbb{SL}$ which are central modulo the next order. We obtain nilpotent groups $\mathbb {SL}/\mathbb{SW}_n$ and
$\mathbb {SL}/\mathbb{SW}^\iinfty_n$, with associated graded groups
\[
\sSW_n\cong\sW_n   \quad \text{ and } \quad\sSW^\iinfty_n \cong \sW^\iinfty_n
\]
\end{thm}
The finitely generated abelian groups  $\sW_n$ and $\sW^\iinfty_n$ associated to the (twisted) Whitney tower filtration on $\LL$ are computed in \cite{CST1,CST2} by identifying the Whitney tower intersection/obstruction theory with the first non-vanishing Milnor invariants and certain higher-order Sato-Levine invariants and higher-order Arf invariants. As an application of these techniques to the setting of string links, Theorem~\ref{thm:nilpotent} will be proven
in Section~\ref{sec:string-links}.

Also proven in Section~\ref{sec:string-links} are the next theorem and subsequent corollary, which follow from
the classification of $\sW^\iinfty_n$, together with the interpretation of the Artin representation
as the ``universal'' Milnor invariant \cite{HL2}. 
\begin{thm} \label{thm:Artin-kernel}
We have $ \SW^\iinfty_n\subset\SJ_n$, and the
 Artin representation $\operatorname{Artin}_n$ induces an epimorphism 
\[
\operatorname{Artin}^\iinfty_{n}\colon\mathbb {SL}/\mathbb{SW}^\iinfty_{n}\twoheadrightarrow\operatorname{Aut}_0(F/F_{n+2})
\]
The kernel is a finite $2$-group, generated by (internal band sums of) iterated Bing-doubles of the figure eight string knot (possibly with some additional trivial strands). In particular, for each $n$ there is an upper bound on the size of this kernel.
\end{thm}
See Figures~\ref{fig:bingstring} and ~\ref{fig:internalbandsum} for the definition of Bing-doubling and internal band sums in the setting of string links.

In fact the kernel of $\operatorname{Artin}^\iinfty_{n}$ can be characterized geometrically in several other ways (the different geometric conditions are explained in Definition~\ref{def:string link}):
\begin{cor}\label{cor:Artin-kernel}
The following subsets of $\mathbb {SL}/\mathbb{SW}^\iinfty_{n}$ are equal to the kernel of $\operatorname{Artin}^\iinfty_{n}$:
\begin{enumerate}
\item The subgroup generated by (internal band sums of) iterated Bing-doubles of a fixed string knot $K_0$ with nontrivial Arf invariant (possibly with some additional trivial strands).
\item The subgroup generated by (internal band sums) of iterated Bing-doubles of all string knots with non-trivial Arf invariant (possibly with some additional trivial strands).
\item The set of equivalence classes of boundary string links.
\item The set of equivalence classes of $\pi_1$-null string links.
\end{enumerate}
\end{cor}

Regarding the other filtrations, we will show in Proposition~\ref{prop:inclusions} that $\SY_n\subseteq \SW_n$. (In fact, in an upcoming paper we will prove that this is an equality, and that the relation of Whitney tower concordance is equivalent to the notion of $C_n$-concordance studied by Meilhan and Yasuhara \cite{MY}.) Summarizing, we see that the filtrations on string links $\SL$ are ordered as follows:
\[
\SY_n\subseteq \SW_n\subseteq \SW^\iinfty_n\subseteq \SJ_n
\]
For $n=1$, all these filtrations are equal to the set $\SL_1:=\SY_1=\SW_1=\SW_1^\iinfty=\SJ_1$ of concordance classes of string links with trivial linking numbers and framings.  

Similarly, for $n=1$ the above filtrations of $\HC(g,b)$ give those homology cylinders which induce the identity homomorphism on first homology $H_1(\Sigma_{g,b})$. 
We write $\HC_1=\Y_1 = \J_1$ for this subgroup.
By taking the complement of a string link in $D^2 \times [0,1]$, one gets a well-known group homomorphism 
\[
\mathrm{C}: \SL(m) \ra\HC(0,m+1)
\]
which takes $\SJ_n$ to $\J_n$, and takes $\SY_n$ to $\Y_n$. In unpublished work \cite{H}, Habegger used the fact that $\Sigma_{g,1} \times [0,1]\cong \Sigma_{0,2g+1} \times [0,1]$ to give a bijection
\[
\HC_1(0,2g+1) \longleftrightarrow \HC_1(g,1)
\]
which is not a group homomorphism but identifies the filtrations $\J_n$ (respectively $\Y_n$) on the two different types of homology cylinders. 
In Section~\ref{sec:map}, we generalize the map $\mathrm{C}$ to another geometric map
\[
\mathrm{H}: \SL_1(2g) \ra \HC_1(g,1) \quad\text{(see Figure~\ref{fig:stringclosure})}
\]
which is not a homomorphism but takes both $\SY_n$ to $\Y_n$ and $\SJ_n$ to $\J_n$ for $n\geq 1$. We prove that 
this map $\mathrm{H}$ agrees with Habegger's bijection pre-composed with $\mathrm{C}$.  Combining results from  \cite{HL2} and \cite{GL} it follows that the induced maps on the associated graded groups $\mathrm{C}_n:\sSJ_n \to \sJ_n$ are group isomorphisms for all $n\geq 1$, and by composing with Habegger's bijection we see that the same is true for $\mathrm{H}_n: \sSJ_n \to \sJ_n$. Here $\sSJ_n$ and $\sJ_n$ are the quotient groups $\SJ_n/\SJ_{n+1}$ and $\J_n/\J_{n+1}$ of the Johnson filtrations. 

As will be explained in Section~\ref{sec:map}, the analogous induced maps for the $Y$-filtrations are not yet fully understood but again the statements for $\mathrm{C}_n$ and $\mathrm{H}_n$ are equivalent:
%Comment: HL2 show Artin rep is surjective so that assoc grade for string links is iso to D_n. GL show assoc graded for homology cylinders is D_n.

\begin{thm}\label{thm:H}
The induced maps $\mathrm{C}_n, \,\mathrm{H}_n: \sSY_n\to \sY_n$ are group isomorphisms for $n \equiv 0,2,3 \mod 4$. In the remaining cases, $\mathrm{C}_{4n+1},\,\mathrm{H}_{4n+1}$ are epimorphisms with finitely generated 2-torsion kernel. 
%(which vanishes conjecturally).
\end{thm} 
The graded group $\sSY_n$ is defined as the quotient of $\SY_n$ by the equivalence relation generated by simple order $n+1$ clasper surgeries. (So for example two string links representing elements in $\SY_n$ are equivalent if and only if they differ by a sequence of concordances and order $n+1$ simple clasper surgeries.)
Similarly $\sY_n$ is defined as the quotient of  $\Y_n$ by the equivalence relations generated by order $n+1$ simple clasper surgeries. 
 
As an aside, we recall that by a theorem of Nielsen \cite{Nie,ZVC}, $\HC(g,b)$ contains the mapping class group of $\Sigma_{g,b}$. This is one source of interest in the filtrations $\J_n$ and $\Y_n$.

Levine had already observed that in $\HC(g,1)$ there is an inclusion $\Y_n\subseteq \J_n$ and he started to study the difference in \cite{L1,L2}. He conjectured the statements of the next theorem which we proved in \cite{CST3}. 

In the following, $\sL=\sL(2 g)$ denotes the free Lie algebra (over the integers) on $2 g$ generators (of degree 1), and $\sL_n$ is the degree $n$ part of $\sL$ (additively generated by $n$-fold brackets of generators, modulo the Jacobi identity and the self-annihilation relations $[X,X]=0$).

\begin{thm}[\cite{CST3}] \label{thm:filtration}
For $n\geq 1$, the inclusion of filtrations $\Y_n\subseteq \J_n$ of $\HC(g,1)$ induces the following exact sequences of associated graded groups:
$$0\to{\sf Y}_{2n}\to {\sf J}_{2n}\to \sL_{n+1}\otimes\z \to 0$$
$$\Z_2^m\otimes \sL_{n+1} \to{\sf Y}_{2n+1}\to {\sf J}_{2n+1}\to 0$$
\end{thm}

Levine did not conjecture that the map $\z^m\otimes \sL_{n+1} \to{\sf Y}_{2n+1}$ is injective, and in fact it is \emph{not} because the \emph{framing relations} introduced in \cite{CST1} are also present in this context, as will be seen in Proposition~\ref{prop:twistedtheta} below. 
%Dividing $\cT_{2n-1}$ by these relations leads to the group $\widetilde{\cT}_{2n-1}$.
Unravelling the odd order case is the main result in this paper:

\begin{thm}\label{thm:improved-filtration}
For $n\geq 1$, there are exact sequences of associated graded groups:
$$0\to\sL_{2n+1}\otimes\z\to {\sf Y}_{4n-1}\to{\sf J}_{4n-1}\to 0$$
$$0\to{\sf K}^{\sf Y}_{4n+1}\to{\sf Y}_{4n+1}\to {\sf J}_{4n+1}\to 0$$
and the kernel ${\sf K}^{\sf Y}_{4n+1}$ fits into the exact sequence $\sL_{n+1}\otimes\z\overset{a_{n+1}}\to{\sf K}^{\sf Y}_{4n+1}\to \sL_{2n+2}\otimes\z\to 0$.
\end{thm}

Theorem~\ref{thm:improved-filtration} will be proven in section~\ref{sec:homology-cylinders}.
The calculation of ${\sf K}^{\sf Y}_{4n+1}$ is thus reduced to the calculation of $\Ker(a_{n+1})$. As explained at the end of Section~\ref{sec:map}, this is the precise analog of the question ``how nontrivial are the higher-order Arf invariants?'' in the setting of Whitney tower filtrations of classical links (compare the $a_{n+1}$ in Theorem~\ref{thm:improved-filtration} with the maps $\alpha_{n+1}$ defined in \cite{CST0}).

\begin{conj}\label{conj:a-maps}
The homomorphisms $a_{n+1}$ are injective for all $n\geq 1$.
\end{conj}

Underlying all of the above results are two key aspects of our recent classification of the Whitney tower filtrations 
%$\sW_n$ and $\sW_n^\iinfty$ 
on classical links (as surveyed in \cite{CST0}): The first aspect is a strengthening of the tree-valued intersection theory for Whitney towers that incorporates the above-mentioned framing relations, which reflect geometric moves that affect obstructions to framing Whitney disks in a Whitney tower \cite{CST1}. Second is the affirmation in \cite{CST3} of a combinatorial conjecture of Levine \cite{L2} which allows for a complete computation of the resulting target groups for the Whitney tower intersection invariants. We briefly touch on the relevance of these points next, with details given in the body of the paper: 

The abelian group $\cT=\cT(m)$ is the additively generated by oriented unitrivalent trees, where univalent vertices are labeled by elements of the index set $\{1,\ldots,m\}$, modulo IHX relations and AS antisymmetry relations $t+(-t)=0$, where $-t$ denotes the tree $t$ with the opposite orientation (see e.g.~\cite{CST1,CST3}). Here the {\em orientation} of a tree is given by cyclic orderings of the edges at all trivalent vertices, and trees are considered up to label-preserving isomorphisms.
The \emph{order} of a tree is the number of trivalent vertices. Since the AS and IHX relations are homogeneous with respect to order, $\cT$ inherits a grading $\cT=\oplus_n\cT_n$, where $\cT_n$ is the free abelian group on order $n$ trees, modulo AS and IHX relations. 

The resolution \cite{CST3} of the Levine Conjecture allows for a complete computation of $\cT_n$ for all $n$ as free abelian groups of known rank in even orders, and having only $2$-torsion of known rank in odd orders.
Analogous groups with rational coefficients were well understood previously, and play a central role in finite type theories (e.g.~\cite{GGP,H,HM,Hab}). The move to integral coefficients is critical for extracting combinatorial/topological information from invariants such as Milnor's link invariants and the closely related
Artin representation.

Using surgeries along tree claspers, Habiro \cite{Hab} constructed surjective realization maps
\[
\theta_n :\cT_n \sra \sY_n
\]
where $\Y_n$ filters either $\HC(0,m+1)$ or $\HC(g,1)$ (which are equivalent for $n\geq 1$ by the above discussion).
As described in \cite{CST1}, the quotients $\wT_n$ of the $\cT_n$ by the framing relations are natural targets for the intersection invariants associated to the obstruction theory of (framed) Whitney towers, and play a crucial role in the computation of $\sW_n$. (See Definition~\ref{def:reduced-tree-group} below for the definition of $\widetilde{\cT}_n$.)

A lot of the work of Section~\ref{sec:homology-cylinders} goes into translating the geometry of the framing relations into the clasper language, culminating in the following result which is used in proving our main Theorem~\ref{thm:improved-filtration} above.

\begin{prop}\label{prop:twistedtheta}
Habiro's realization map $\theta_{2n+1}$ factors through the quotient $\wT_{2n+1}$ for $n\geq 1$.
\end{prop}

In even orders $\widetilde{\cT}_{2k}=\cT_{2k}$ by definition (since the framing relations only occur in odd orders), 
and the main result of \cite{CST3} also computes the groups $\widetilde{\cT}_n$ for all $n$. These computations are essential to the proofs of Theorem~\ref{thm:H} and Theorem~\ref{thm:improved-filtration}. Further implications of \cite{CST3} appear
in the detailed description of the relationship between the string link and homology cylinder settings given in Section~\ref{sec:map}; as well as in the characterization of the kernel of $\operatorname{Artin}^\iinfty_{n}$ given by Theorem~\ref{thm:Artin-kernel} and Corollary~\ref{cor:Artin-kernel} as explained in Section~\ref{sec:string-links}. 
%COULD SAY MORE ABOUT LC HERE, BUT D AND D' INTRODUCED PRECISELY IN SECTION 3 WHERE NEEDED...

The paper is organized as follows. In section~\ref{sec:string-links}, we present basic Whitney tower techniques in the setting of string links and prove Theorems~\ref{thm:nilpotent} and \ref{thm:Artin-kernel}, as well as Corollary~\ref{cor:Artin-kernel}. Section~\ref{sec:homology-cylinders} is devoted to clarifying and proving Theorem~\ref{thm:improved-filtration}, and is mostly occupied with proving the just-stated Proposition~\ref{prop:twistedtheta}.  Finally in section \ref{sec:map}, we compare string links with homology cylinders using Habegger's map, proving Theorem~\ref{thm:H} and explaining the connection between Conjecture~\ref{conj:a-maps} above and higher-order Arf invariants for links \cite{CST0,CST1}.

{\bf Acknowledgments:} We thank Max-Planck-Institut-f\"ur-Mathematik for their generous support. Parts of this paper were conceived and/or written while the first two authors were visting the third at MPIM in the Fall of 2010.

\section{The string link Artin Representation and twisted Whitney towers}\label{sec:string-links}
Recall from e.g.~\cite{CST1} that a Whitney tower $\cW\subset B^4$ bounded by a link $L\subset S^3=\partial B^4$ consists
of properly immersed (\emph{order zero}) disks bounded by the link components, together with iterated collections of Whitney disks of increasing \emph{order}
pairing up lower order intersections. An order $n$ Whitney tower $\cW$ has all intersections of order less than $n$ paired, and each unpaired order $n$ intersection point in $\cW$ determines an order $n$ generator of $\cT_n$ which sits as a subset of $\cW$, with each trivalent vertex contained in a Whitney disk interior (which induces the vertex-orientation), and each edge a sheet-changing path across a Whitney disk boundary. The univalent vertices of generators of $\cT_n$ are labeled by the components of $L$, and
summing over all order $n$ intersections determines an intersection invariant $\tau_n(\cW)\in\cT_n$, with the vanishing of $\tau_n(\cW)$
sufficient to conclude that $L$ bounds an order $n+1$ Whitney tower. 

An order $n$ \emph{twisted} Whitney tower only differs in that
certain Whitney disks (of order at least $n/2$ for $n$ even) are allowed to be twisted rather than framed, and such an order $n$ twisted
$\cW$ has an analogous intersection invariant $\tau^\iinfty_n(\cW)\in\cT^\iinfty_n$, where (in even orders) the abelian group $\cT^\iinfty_n$ is a quadratic refinement of $\cT_n$, defined using trees which may have an $\iinfty$ symbol labeling a univalent vertex. (In section 3, we will introduce \emph{twisted claspers}, which are allowed to have a null-homotopic twisted leaf.)

This obstruction theory was used in \cite{CST1} to show that the sets  $\W_n^{(\iinfty)}$ 
of links bounding order $n$ (twisted) Whitney towers, modulo order $n+1$ (twisted) Whitney tower concordance, form groups 
$\sW_n^{(\iinfty)}$ under component-wise band sums.  As surveyed in \cite{CST0}, the sequence of papers \cite{CST1,CST2,CST3} computes the $\sW_n^{(\iinfty)}$ in terms of Milnor invariants together with \emph{higher-order Arf invariants}
taking values in finite $2$-groups, with the caveat that
the image of the higher-order Arf invariants is not known. (The framed case requires \emph{higher-order Sato-Levine invariants} as well, which are certain mod $2$ projections of Milnor invariants of higher order.)

Let $\sigma$ be a string link with $m$ strands embedded in $B=D^2\times[0,1]$. By a theorem of Stallings \cite{Sta}, the inclusions 
$(D^2\setminus \{\mbox{$m$ points}\})\times\{i\}\hookrightarrow B\setminus \sigma$ for $i=0,1$ induce isomorphisms on all lower central quotients of the fundamental groups. In fact, the induced automorphism of the lower central quotients 
$F/F_n$ of the free group $F=\pi_1(D^2\setminus \{\mbox{$m$ points}\})$ is explicitly characterized by conjugating the meridianal generators by the string link longitudes. Let $\operatorname{Aut}_0(F/F_n)$ consist of those automorphisms of $F/F_n$ which are defined by conjugating each generator and which fix the product of generators. This leads to the \emph{Artin representation} 
$$\operatorname{Artin}_n\colon\mathbb{SL}\to \operatorname{Aut}_0(F/F_{n+2})$$ where $\mathbb{SL}$ is the set of concordance classes of framed string links. See \cite{HL,HL2} for more details. Using the correspondence between string links and homology cylinders over planar surfaces, we define the \emph{Johnson filtration}  by letting $\SJ_n=\Ker\{\operatorname{Artin}_n\colon\mathbb{SL}\to \operatorname{Aut}_0(F/F_{n+2})\}$. 

The set of string links has a well-defined monoid structure given by stacking, which, modulo concordance, becomes a (nonabelian \cite{LD}) group $\mathbb{SL}$. The Whitney tower filtrations have analogous definitions in this context, giving rise to filtrations $\mathbb{SW}_n$ and $\mathbb{SW}_n^\iinfty$ of $\mathbb{SL}$. More precisely, given a string link $\sigma$, there is a standard closure $L_\sigma\subset S^3$ obtained by connecting the endpoints by fixed arcs in the boundary of $B$ and taking a standard embedding $B\subset S^3$ (thinking of $B=B\times\{0\}\subset B\times[0,1]$, with $S^3=\partial (B\times[0,1])$). Then a string link $\sigma$ is by definition in $\mathbb{SW}_n^{(\iinfty)}$ if $L_\sigma$ is the boundary of a (twisted) Whitney tower of order $n$ in $B\times[0,1]\cong B^4$. 

 For the statement of Theorem~\ref{thm:Artin-kernel} we need the following operation of Bing-doubling on a framed string link, represented by a tangle $T$. Here the four parallel arcs in the tangle $T'$  that replace the arc in $T$ are assumed to follow the given framing.
\begin{figure}[h]
\begin{center}
$$
\begin{minipage}{.45\linewidth}
\includegraphics[width=\linewidth]{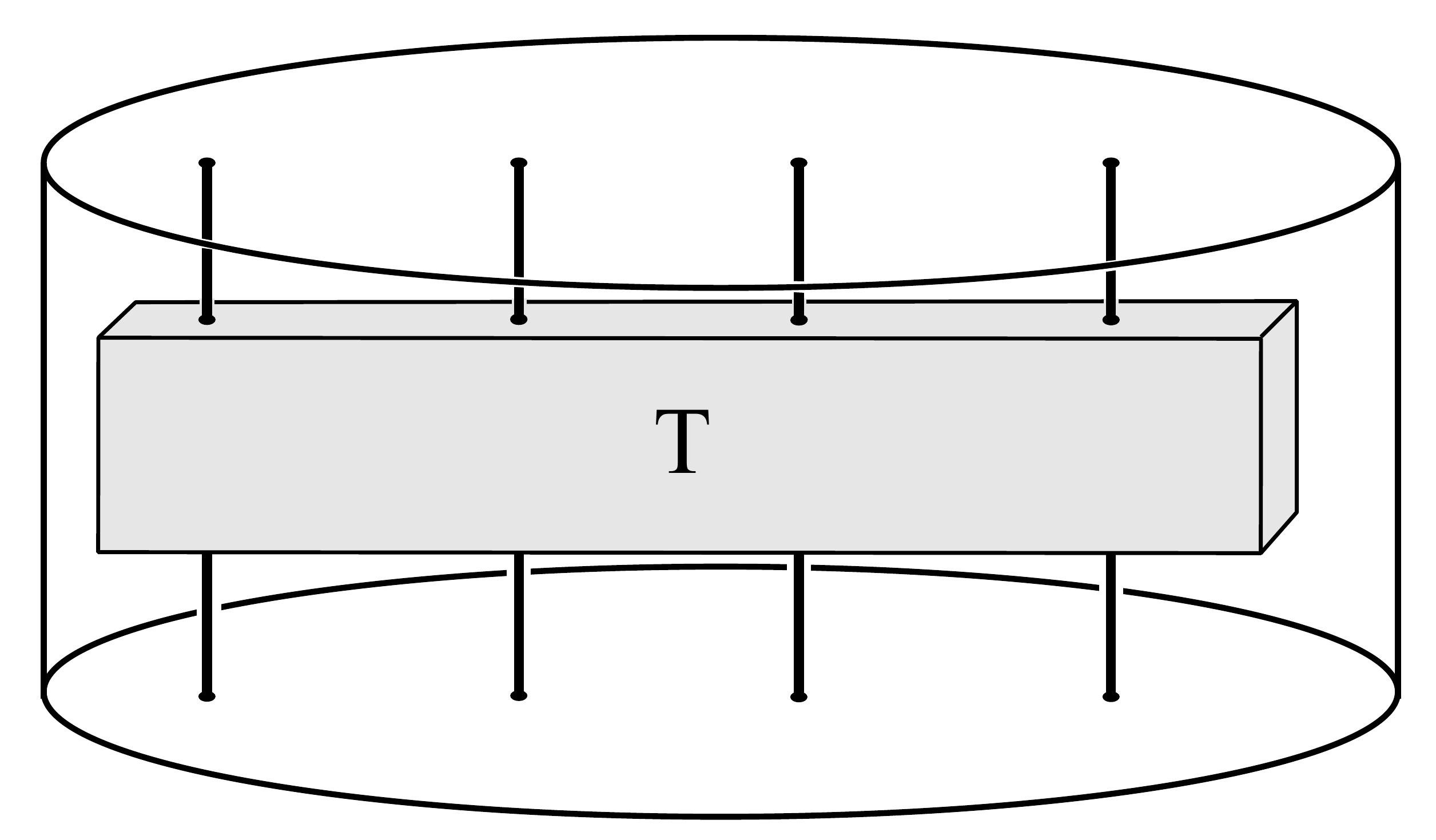}
\end{minipage}
\longrightarrow
\begin{minipage}{.45\linewidth}
\includegraphics[width=\linewidth]{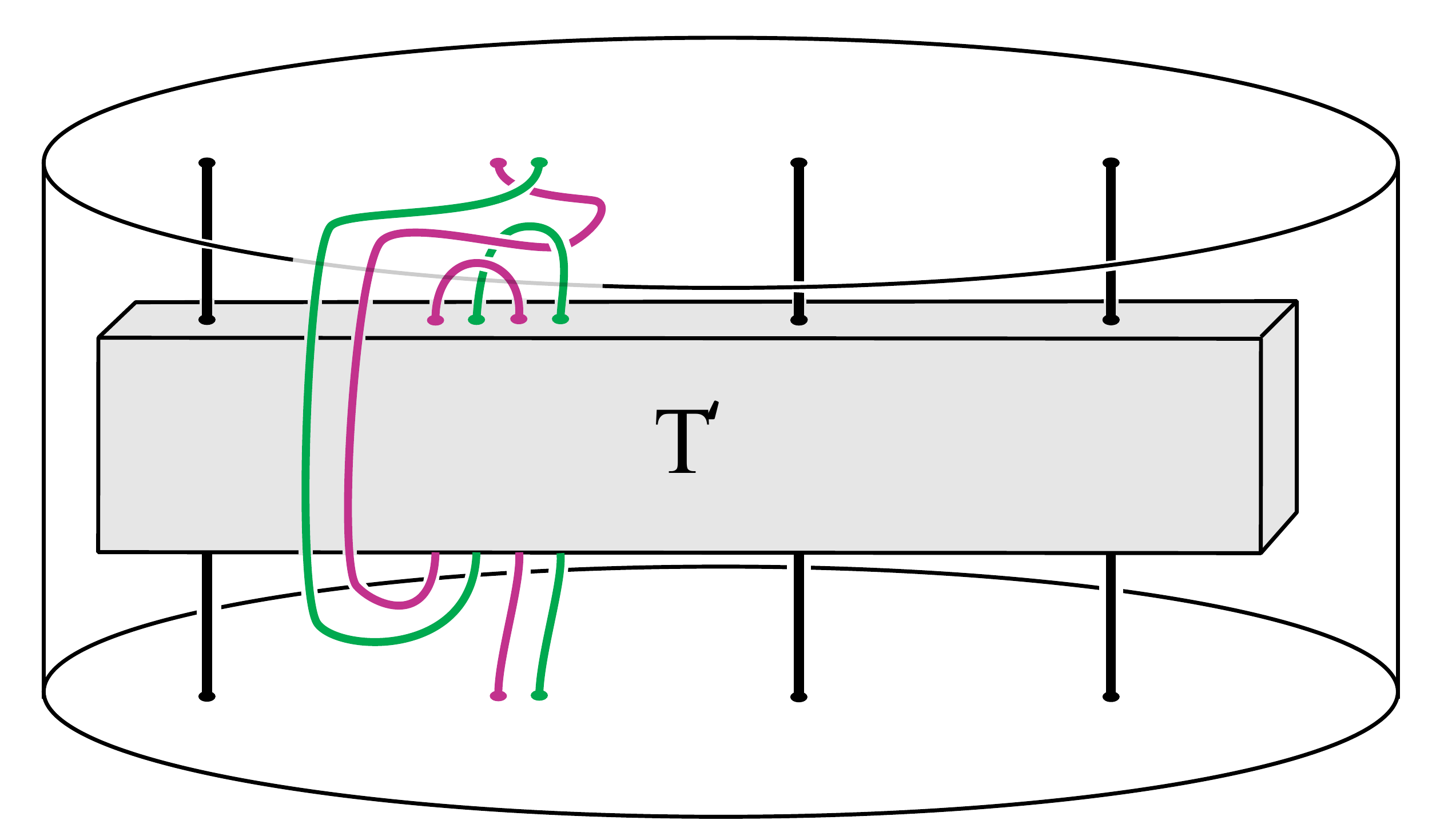}
\end{minipage}
$$
\caption{Bing-doubling a component of a framed string link. }\label{fig:bingstring}
\end{center}
\end{figure}

In the statement of Corollary~\ref{cor:Artin-kernel}, we used the following notions.
\begin{defn} \label{def:string link}
A string link $\sigma$ is a \emph{boundary string link} if the components of the standard closure $L_\sigma$ bound disjoint surfaces in the $3$--ball $B$; and a 
\emph{$\pi_1$-null string link} is a string link $\sigma$ whose standard closure $L_\sigma$ bounds a surface $\Sigma$ in the $4$--ball $B^4=B\times[0,1]$ such that $\pi_0(L_\sigma)\to\pi_0(\Sigma)$ is a bijection and
which for which there is a push-off inducing the trivial homomorphism $\pi_1(\Sigma)\to\pi_1(B^4\setminus\Sigma)$. Note that Bing-doubling preserves boundary links: the standard closure of the string link in Figure~\ref{fig:bingstring} is designed to look like Figure 21 in \cite{CST2}, for which there are evident disjoint Seifert surfaces. Given a string link, we define an internal band sum between two of the components $i,j$ as follows. Choose an arc in the $3$--ball $B$ between the $i$ and $j$ strands, which will guide the band-sum.
Add the boundary arc which makes the standard closure of component $j$, and push this closure into $B$ slightly. This gives us a tangle where strand $j$ is now a closed component. Finally do surgery along the chosen arc. In the case that the string link is a boundary string link, choose the arc to avoid the bounding surfaces. Then by construction the internal band sum will also be a boundary string link. An example is given in Figure~\ref{fig:internalbandsum}.
\end{defn}

\begin{figure}
\begin{center}
\subfloat[]{\includegraphics[width=.45\linewidth]{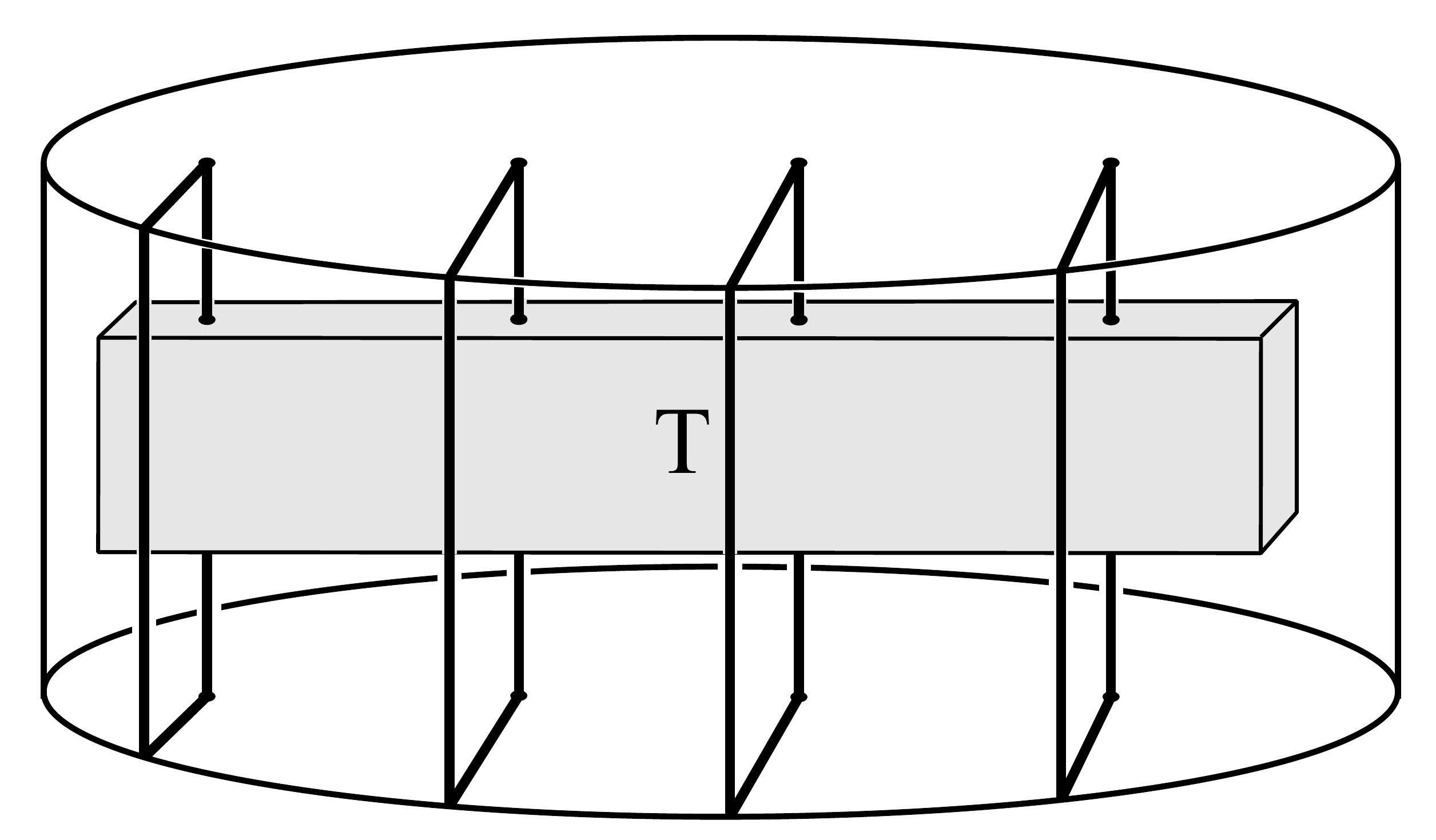}}\hfill\subfloat[]{\includegraphics[width=.45\linewidth]{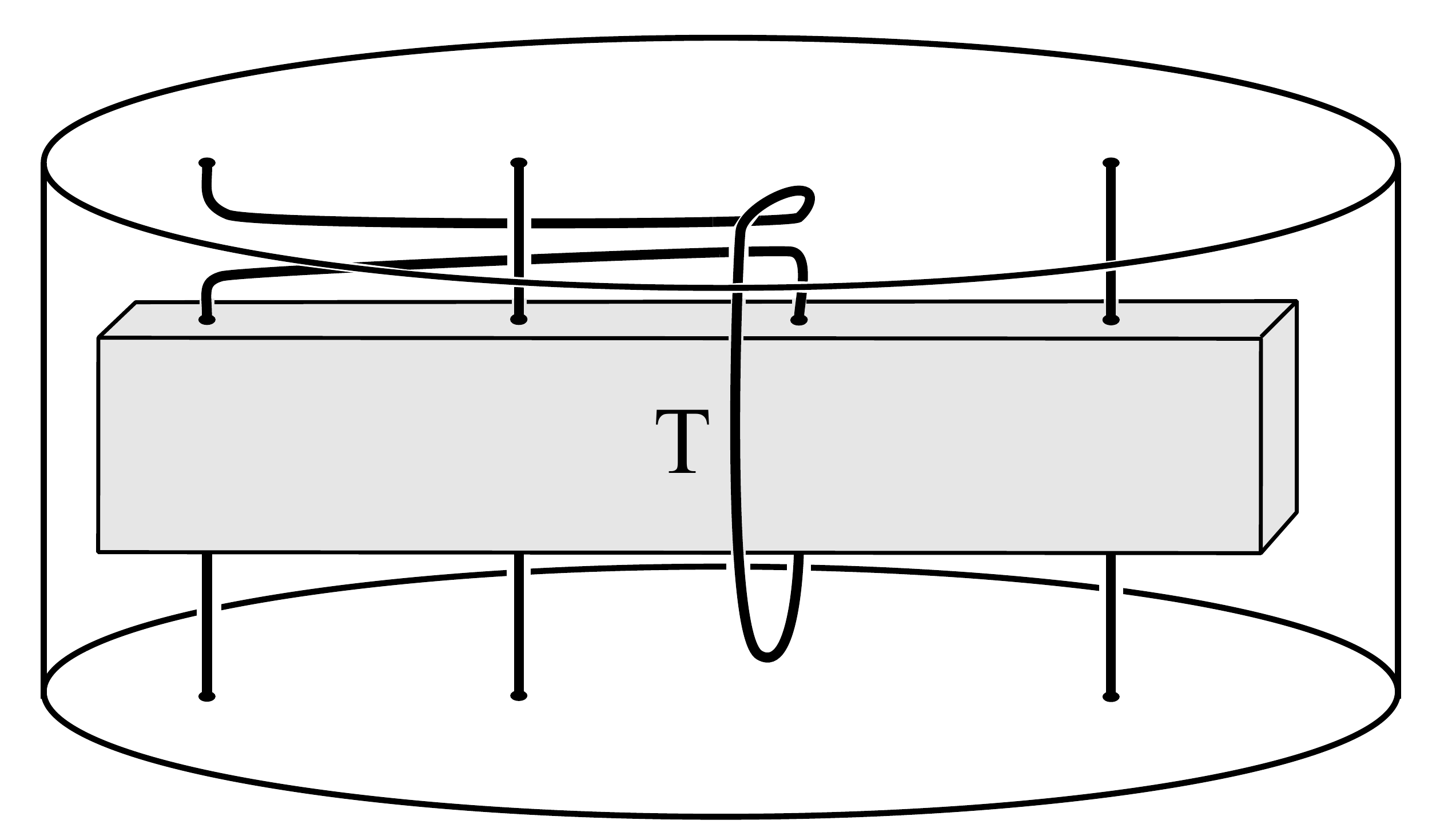}}
\caption{The standard closure of a string link, and an internal band sum between the first and third component. If the link in (A) bounds disjoint surfaces in the given $3$--ball, then the (B) link is a boundary string link, since the surfaces bounding the first and third components on the left can be joined by a well-chosen band.}\label{fig:internalbandsum}
\end{center}
\end{figure}

\subsection{Proof of Corollary~\ref{cor:Artin-kernel}}\label{subsec:proof-corollary-Artin-kernel} The proof is as follows: All of the four sets listed in the statement of Corollary~\ref{cor:Artin-kernel} are certainly contained in the kernel of $\operatorname{Artin}^\iinfty_{n}$, since Milnor invariants vanish on all of these classes, and the 
%$\operatorname{Artin}_{n}$
order $n$ Artin representation corresponds to the universal Milnor invariants of order $n$ (length $n+2$) as described in section~3 of \cite{HL2}. On the one hand, these sets are increasing in size, and on the other hand, the smallest set is equal to the kernel of $\operatorname{Artin}^\iinfty_{n}$ by the proof of Theorem~\ref{thm:Artin-kernel} below.

%===============================
\subsection{Proof of Theorems~\ref{thm:nilpotent} and \ref{thm:Artin-kernel}}\label{subsec:proof-thm-Artin-ker}

String links $\sigma$ and $\sigma '$ are said to be \emph{order $n$ (twisted) Whitney tower concordant} if $\sigma\subset B\times\{0\}$ cobounds an order $n$ (twisted) Whitney tower in $B\times I$ with $\sigma'\subset B\times\{1\}$ (together with product arcs in $\partial B\times I$ connecting the boundaries $\partial \sigma$ and $\partial \sigma'$). It is easy to see that $\sigma$ is order $n$ (twisted) Whitney tower concordant to the identity string link if and only if $\sigma\in\mathbb{SW}_n^{(\iinfty)}$ (which by definition means that $L_\sigma\in\mathbb{W}_n^{(\iinfty)}$).
We observe that all results from \cite{CST1,CST2} for Whitney tower filtrations of link concordance also apply to string link concordance. In particular, the obstruction theory of \cite{CST1} implies that if $\sigma$ bounds an order $n$ (twisted) Whitney tower $\cW$ in $B\times[0,1]$ with $\tau_n^{(\iinfty)}(\cW)=0\in\cT_n^{(\iinfty)}$, then $\sigma\in\mathbb{SW}_{(n+1)}^{(\iinfty)}$. And similarly, if $\cW$ is a (twisted) order $n$ Whitney concordance from $\sigma$ to the identity with $\tau_n^{(\iinfty)}(\cW)=0\in\cT_n^{(\iinfty)}$, then $\sigma\in\mathbb{SW}_{(n+1)}^{(\iinfty)}$.

Clearly $\mathbb{SW}_n$ and $\mathbb{SW}_n^\iinfty$ are closed under multiplication, since order $n$ (twisted) Whitney towers on the factors of a product stack together to form an order $n$ (twisted) Whitney tower on the product. 
Also, the reflection $\overline{\cW}$ of any order $n$ (twisted) Whitney tower $\cW$ bounded by $\sigma$ is an order $n$ (twisted) Whitney tower bounded by the concordance inverse $\bar{\sigma}$ of $\sigma$. Note that Section~3 of \cite{CST1} implies that $\tau_n^{(\iinfty)}(\cW)=-\tau_n^{(\iinfty)}(\overline{\cW})$ for any such  $\cW$ and $\overline{\cW}$.

The normality and centrality statements will follow from the following basic observation:
If $\sigma_1$ cobounds an order $n$ Whitney tower $\cW_1$ with the identity string link $\sigma_{\id}$, then for any string link $\sigma_2$, the conjugate string link $\sigma_2\sigma_1\bar{\sigma_2}$ cobounds an order $n$ Whitney tower $\cW_2$ with $\sigma_{\id}$, such that $\tau_n(\cW_1)=\tau_n(\cW_2)\in\cT_n$. 
Such a $\cW_2$ is constructed by first stacking $\cW_1$ between product concordances of $\sigma_2$ and $\bar{\sigma_2}$ to get an order $n$ Whitney tower concordance from $\sigma_2\sigma_1\bar{\sigma_2}$ to $\sigma_2\sigma_{\id}\bar{\sigma_2}=\sigma_2\bar{\sigma_2}$; then composing with any concordance $\cW_0$ from $\sigma_2\bar{\sigma_2}$ to $\sigma_{\id}$. (The concordance $\cW_0$ is a Whitney tower of arbitrarily high order since the disjointly embedded disks are free of singularities.) 
Then $\tau_n(\cW_2)=\tau_n(\cW_1\cup\cW_0)=\tau_n(\cW_1)$, and in fact $\cW_2$ and $\cW_1$ have identical \emph{intersection forests} (disjoint unions of signed trees, see \cite{CST1}). This observation also holds in the twisted setting, and immediately implies that $\mathbb{SW}_n$ and $\mathbb{SW}_n^\iinfty$ are normal subgroups of $\mathbb{SL}$.

Now suppose that $\sigma_1\in\mathbb{SW}_n$. We claim $\sigma_1\sigma_2\bar{\sigma_1}\bar{\sigma_2}\in \mathbb{SW}_{n+1}$ for any $\sigma_2$. To see this let $\cW_1$ be any order $n$ Whitney tower concordance from 
$\sigma_1$ to $\sigma_{\id}$, and $\overline{\cW_1}$ be the reflection cobounded by $\bar{\sigma_1}$ and $\sigma_{\id}$. As in the observation of the previous paragraph, $\overline{\cW_1}$ extends to an order $n$ Whitney tower concordance 
$\cW_2$
from $\sigma_2\bar{\sigma_1}\bar{\sigma_2}$ to $\bar{\sigma_1}$ such that $\tau_n(\cW_2)=\tau_n(\overline{\cW_1})$. Reparametrizing $\cW_1$ (composing with a product concordance) and stacking  with $\cW_2$ yields an order $n$ Whitney tower concordance $\cW$ from $\sigma_1\sigma_2\bar{\sigma_1}\bar{\sigma_2}$ to $\sigma_{\id}$ with 
$$
\tau_n(\cW)=\tau_n(\cW_1)+\tau_n(\cW_2)=\tau_n(\cW_1)+\tau_n(\overline{\cW_1})=0\in\cT_n
$$
so $\sigma_1\sigma_2\bar{\sigma_1}\bar{\sigma_2}$ bounds an order $n+1$ Whitney tower
by the obstruction theory of \cite{CST1}.
This implies the nilpotence of the quotient groups $\mathbb {SL}/\mathbb{SW}_n$, and the analogous argument applies in the twisted setting. 

In order to see that the string link groups are isomorphic to the corresponding link groups, 
notice we have a surjective map $\sigma\mapsto L_\sigma$ from $\mathbb {SL}$ to the set $\mathbb L$ of concordance classes of links in the $3$--sphere, which induce surjections $\mathbb {SL}/\mathbb{SW}^{(\iinfty)}_n\twoheadrightarrow \mathbb L/\mathbb W^{(\iinfty)}_n$. 
The image $L_\sigma\subset S^3$ of any $\sigma$ in the kernel bounds a (twisted) Whitney tower of order $n+1$ in $B^4$, which becomes a (twisted) Whitney tower of order $n+1$ bounded by $\sigma$ via the operation of choosing a disk-basing of $L_\sigma$ \cite{HL} and cutting $S^3$ open along a thickening of a wedge of arcs connecting the components of $L_\sigma$. We thus conclude that $\sSW^{(\iinfty)}_n=\sW^{(\iinfty)}_n$ as claimed in Theorem~\ref{thm:nilpotent}.

To see that $\SW^\iinfty_n \subset \SJ_n$, i.e.~that $\Artin_{n}^\iinfty(\mathbb{SW}^\iinfty_{n})$ is trivial, we note that by computations in \cite[Thm.5]{CST2}, the longitudes of an element of $\mathbb{SW}^\iinfty_{n}$ all lie in the $(n+1)$th term of the lower central series, so that conjugation by such an element is trivial modulo the $(n+2)$th term. 

If $\Artin_n^\iinfty$ vanishes on $\sigma$, then all Milnor invariants of $\sigma$ up to order $n-1$ vanish. Since these Milnor invariants together with higher-order Arf invariants detect all nontrivial elements of $\mathbb{SL}/\mathbb{SW}^\iinfty_n$ (see \cite[Cor.1.16]{CST1}), the kernel of $\Artin_n^\iinfty$ is generated by links introduced in Lemma~12 of \cite{CST2} which have trivial Milnor invariants (because they are boundary links) but which realize all values of the higher-order Arf invariants. These are indeed (internal band sums of) iterated Bing-doubles of the figure-eight knot. See Figures~\ref{fig:bingstring} and ~\ref{fig:internalbandsum} for the Bing-doubling and internal band sum construction in the setting of string links.

The fact that the kernel is a $2$-group can be shown as follows. Let $\sigma$ be a string link in the kernel of $\Artin^{\iinfty}_n$.
For any order $k$ twisted Whitney tower $\cW$ bounded by $\sigma$ with $k<n$, it follows from \cite[Thm.1.13 and Prop.1.14]{CST1}
 (or \cite[Thm.7 and Prop.8]{CST2}), that $\tau^\iinfty_k(\cW)$ vanishes if $k\equiv 0,1,3\mod 4$ and $\tau^\iinfty_k(\cW)$ is $2$-torsion if $k\equiv 2\mod 4$. So if $\cW_\sigma$ is any order $2$ twisted Whitney tower bounded by $\sigma$, then $\sigma^2$ bounds $\mathcal W_{\sigma^2}$ formed by stacking two copies of $\cW_\sigma$, and $\tau^\iinfty_2(\mathcal W_{\sigma^2})=2\tau^\iinfty_2(\mathcal W_{\sigma})=0$. Hence $\sigma^2\in\mathbb{SW}^\iinfty_6$. But then $\tau^\iinfty_6(\mathcal W_{\sigma^4})=0$, where $\mathcal W_{\sigma^4}$ is the double of an order $6$ twisted Whitney tower bounding $\sigma^2$. Hence $\sigma^4\in\mathbb{SW}^\iinfty_{10}$, etc. In particular, if $4k-2>n$, then $\sigma^{2^k}=1\in \mathbb{SL}/{\mathbb{SW}^\iinfty_{n+1}}$. $\hfill\square$

%=====================================================================

\section{Comparing filtrations of the group of homology cylinders}\label{sec:homology-cylinders}
This section clarifies and proves Theorem~\ref{thm:improved-filtration} from the introduction, which significantly sharpens Levine's description of the relationship between the algebraically defined
Johnson filtration and the geometrically defined filtration based on the Goussarov-Habiro
theory of finite type $3$--manifold invariants and clasper surgery. As mentioned in the introduction, 
the central idea of the proof is to show that the realization map $\theta_{2n-1}\colon\cT_{2n-1}\to {\sf Y}_{2n-1}$ from Theorem~\ref{thm:Levine} induces a map $\widetilde{\theta}_{2n-1}$ on the quotient $\widetilde{\cT}_{2n-1}$ of $\cT_{2n-1}$ by the framing relations, as stated in Proposition~\ref{prop:twistedtheta}.
We shall recall the definition of $\widetilde{\cT}_n$ in Definition~\ref{def:reduced-tree-group} and prove Proposition~\ref{prop:twistedtheta} in Section~\ref{subsec:twistedtheta}.
%the key new ingredient is Proposition~\ref{prop:twistedtheta}, proven in this section, which states that Habiro's surgery map $\theta_{2n+1}\colon\cT_{2n+1}\to{\sY}_{2n+1}$ induces a well-defined map  $\widetilde{\theta}_{2n+1}\colon\widetilde{\cT}_{2n+1}\to {\sY}_{2n+1}$.
This also leads to the formulation of higher-order Arf invariants
defined for certain homology cylinders (Remark~\ref{rem:higher-order-invariants}) analogous to those defined in the Whitney tower filtrations on links.

We briefly introduce the relevant concepts relating to homology cylinders. For background and related developments  the reader is referred to the
recent survey \cite{HabMass2} and its references.  

 Let $\Sigma_{g,1}$ denote the compact oriented surface of genus $g$ with one 
boundary component. Following Garoufalidis and Levine \cite{GL, L1}, we define a \emph{homology cylinder} over $\Sigma_{g,1}$ to be a compact oriented $3$--manifold $M$ equipped with a homeomorphism 
\[
m\colon \partial\left(\Sigma_{g,1}\times[-1,1]\right)\to \partial M
\]
  such that the two maps $m_{\pm}=m|_{\Sigma_{g,1}\times \pm 1}$, when composed with the inclusion $\partial M\subset M$, give an isomorphism on homology. This definition differs slightly from Habiro and Massuyeau \cite{HabMass}, who reserve the term ``homology cylinder" for those homology cylinders with a trival action on $H_*(\Sigma_{g,1})$. Conveniently, the two definitions coincide for positive order in the Johnson filtration, which is all we will consider anyway. Two homology cylinders $M_0$ and $M_1$ are said to be \emph{homology cobordant} if there is a compact oriented $4$--manifold $W$ with $\partial W=M_0\cup_{\partial} (-M_1)$, such that the inclusions $M_i\hookrightarrow W$ are homology isomorphisms. This defines an equivalence relation on the set of homology cylinders. 
Let $\HC(g,1)$ be the set of homology cylinders up to homology cobordism over $\Sigma_{g,1}$. $\HC(g,1)$ is a group via the ``stacking" operation.

Extending the Johnson homomorphisms on the mapping class group \cite{John}, Garoufalidis and Levine \cite{GL,L1} introduced an Artin-type representation
$\HC(g,1)\to A_0(F/F_{n+1})$ where $F$ is the free group on $2g$ generators, and $A_0(F/F_{n+1})$ is the group of automorphisms $\phi$ of $F/F_{n+1}$ which lift to an endomorphism  $\tilde{\phi}\colon F\to F$ which fixes the product $[x_1,y_1]\cdots [x_g,y_g]$ modulo $F_{n+2}$. Here $\{x_i,y_i\}_{i=1}^g$ is a standard symplectic basis for $\pi_1\Sigma_{g,1}\cong F$, and by benign abuse of notation we use the same letters for elements in $F/F_{n+1}$.  The \emph{Johnson filtration} (also known as the \emph{relative weight} filtration) of $\HC(g,1)$ is defined by $\mathbb {J}_n:=\Ker \{\HC(g,1)\to A_0(F/F_{n+1})\}$.  Define the associated graded group ${\sf J}_n=\mathbb J_n/\mathbb J_{n+1}$. In \cite{GL} it shown that ${\sf J}_n\cong\sD_n$, where
$\sD_n$ is the kernel of the bracket map $\sL_1\otimes\sL_{n+1}\to\sL_{n+2}$ with $\sL_k$ the degree
$k$ component of the free Lie algebra (over $\Z$) on the rank $2g$ free abelian group $F/F_2\cong H_1(\Sigma_{g,1})$.

On the other hand, there is a filtration of $\HC(g,1)$ related to the Goussarov-Habiro theory of finite type $3$--manifold invariants \cite{GGP,Hab}: Define the relation of $n$-equivalence (also known as 
$A_n$-equivalence or $Y_n$-equivalence) to be generated by the following move: $M\sim_n M'$ if $M'$ is 
diffeomorphic to $M_C$, surgery on some connected graph clasper $C\subset M$ whose associated graph has at least  $n$ internal vertices. 
The \emph{$Y$-filtration} of $\HC(g,1)$ is defined by taking $\mathbb Y_n$ to be the subgroup of all homology cylinders $n$-equivalent to the trivial product $\Sigma_{g,1}\times I$, and we denote the associated graded groups by ${\sf Y}_n:=\mathbb Y_n/\sim_{n+1}$. 

As Levine observes, $\mathbb Y_n\subset \mathbb J_n$, giving rise to a map of associated graded groups ${\sf Y}_n\to {\sf J}_n$. Rationally, Levine showed that this map is an isomorphism, and in fact both ${\sf Y}_n$ and ${\sf J}_n$ are rationally isomorphic with $\mathcal T_{n}=\mathcal T_n(2g)$. This is implied by the following theorem, proven in the pair of papers \cite{L1,L2}, and stated in the introduction of \cite{L2}. It uses a surjective clasper surgery map $\theta_n\colon {\cT}_n\to {\sf Y}_n$ introduced by Habiro in \cite{Hab}, and further elucidated in \cite{HabMass} (Remarks 6.6, 7.8).

\begin{thm}[Levine] \label{thm:Levine} 
 For $n>1$, there is a commutative diagram\\
\centerline{$\xymatrix{
{\cT}_{n}\ar@{->>}^{\theta_n}[r]\ar@/_1pc/[rrr]_{\eta_{n}}&{\sf Y}_n\ar[r]& {\sf J}_n\ar[r]^\cong &{\sf D}_n
}.$}  Moreover, for all $n> 1$, $\eta_n$, and hence all of the maps, are rational isomorphisms.
\end{thm}

Here, and throughout this section, the generators of $\cT_{n}$ have univalent vertices labeled by
elements of $H_1(\Sigma_{g,1})$, with multi-linearity relations for labels included along with the IHX and antisymmetry relations.  The realization map $\theta_n$ is defined by sending a tree  to surgery on $\Sigma_{g,1}\times I$ along a clasper with the corresponding tree-type; and the combinatorially defined map $\eta_n$ essentially sums over all choices of roots on order $n$ trees, using the correspondence between rooted binary trees and non-associative bracketings.  
Note that terminology and notation varies throughout the literature. Translation from the notation of Levine's more recent work \cite{L1,L2,L3} to ours is as follows: $\cA^t_n=\cT_n$, 
$\cF^w_n=\mathbb J_n$, $\cG^w_n=\sJ_n$, $\cF^Y_n=\mathbb Y_n$, $\cG^Y_n=\sY_n$.
By linearity, the labels on univalent vertices of the generators of $\cT_n$ can be assumed to be singletons from the set $\{x_i,y_i\}_{i=1}^g$; this corresponds to the geometric fact that clasper leaves can be ``split''
modulo higher-order equivalence.  

The maps in Theorem~\ref{thm:Levine} are not isomorphisms and analyzing the above diagram over $\Z$ is the main subject of this section. With the same goal in mind,
Levine was led to study \cite{L2,L3} the free \emph{quasi}-Lie algebra $\sL'$, in which the self-annihilation relation $[X,X]=0$  in $\sL$ is replaced by the weaker antisymmetry relation $[X,Y]=-[Y,X]$ in $\sL'$. In \cite{L2} Levine conjectures that a map $\eta'_n\colon\cT_n\to{\sD}'_n$ is an isomorphism, where $\sD_n'$ is the kernel of the quasi-Lie bracket map $\sL_1'\otimes\sL'_{n+1}\to\sL'_{n+2}$. This map is analogous to $\eta_n$ in Theorem~\ref{thm:Levine} above and we proved this ``Levine conjecture'' in \cite{CST3}. Our analysis below is an application of this result.

It is interesting to observe that with this new notion of a quasi-Lie algebra, Conjecture~\ref{conj:a-maps} from the introduction implies an exact sequence 
\[
0\to \mathbb Z_2\otimes {\sf L}'_{2n+2}\to{\sf Y}_{4n+1}\to {\sf J}_{4n+1}\to 0
\]
This uses the following split exact sequence from \cite{L3}, where the left map sends $X$ to $[X,X]$, and the right map is the natural projection:
$$
0\to\mathbb Z_2\otimes {\sf L}_{k}\to\sL'_{2k}\to \sL_{2k}\to 0
$$

Notice that Conjecture~\ref{conj:a-maps} is not quite correct when $n=0$. Habiro showed that there is a non-canonical isomorphism
\[
{\sf K}^{\sf Y}_1\cong (\z\otimes {\sf L}'_2)\oplus \z,
\]
 which has an extra $\z$-summand.

\subsection{Proof of Theorem~\ref{thm:improved-filtration}}\label{subsec:proof-thm-improved-filtration}

Recall from the statement of Theorem~~\ref{thm:improved-filtration} that ${\sf K}^{\sf Y}_{n}$ is the kernel of the map ${\sf Y}_n\to {\sf J}_n$ induced from the inclusion $\Y_n\subset \J_n$ on the associated graded groups. 
We will need a homomorphism $\kappa_{2n-1}\colon{\sf K}^{\sf Y}_{2n-1}\to\z\otimes{\sf L}_{n+1}$ defined as follows. Suppose $[M]\in{\sf K}^{\sf Y}_{2n-1}$ is an equivalence class of homology cylinders up to $\sim_{2n}.$ Then $M\in \mathbb J_{2n}$. So $M$ determines an element of ${\sf J}_{2n}\cong{\sf D}_{2n}$. Applying the right-hand map in the top sequence of Theorem~\ref{thm:filtration}
%?explain/define $s\ell_{2n}$ here or later(better)? 
yields an element of $\z\otimes{\sf L}_{n+1}$.
This does not depend on the choice of representative, for if $M'\sim_{2n}M$, then $M'\cdot M^{-1}\in\mathbb Y_{2n}$ and $[M'\cdot M^{-1}]\mapsto 0\in \z\otimes {\sf L}_{n+1}$ by the exactness of the top sequence in Theorem~\ref{thm:filtration}. So $M$ and $M'$ both map to the same element of $\z\otimes{\sf L}_{n+1}$. 

The following result will be proven in section~\ref{subsec:proof-prop-ksurjective}.
\begin{prop}\label{prop:ksurjective}
The homomorphism $\kappa_{2n-1}\colon {\sf K}^{\sf Y}_{2n-1}\to \z\otimes \sL_{n+1}$ is surjective.
\end{prop}
\begin{rem}\label{rem:higher-order-invariants}
Regardless of whether or not Conjecture~\ref{conj:a-maps} is true, $a_{n+1}$ induces an isomorphism $\overline{a}_{n+1}$ on $(\mathbb Z_2\otimes {\sf L}_{n+1})/\Ker a_{n+1}$, and in direct analogy with the higher-order Arf invariants defined for links bounding certain Whitney towers
\cite{CST0}, one can define higher-order Arf invariants for homology cylinders representing elements of ${\sf K}^{\sf Y}_{4n+1}$ by inverting $\overline{a}_{n+1}$. This analogy with \cite{CST0} can be further extended by defining higher-order Sato-Levine invariants for homology cylinders using the homomorphism $\kappa_{2n-1}$.\end{rem}

{\bf Proof of Theorem~\ref{thm:improved-filtration}  (order $4n-1$ case)}\\ We claim there is a commutative diagram of short exact sequences:
$$
\xymatrix{
\z\otimes{\sf L}_{2n+1}\ar@{->>}[d]\ar@{>->}[r]&\widetilde{\cT}_{4n-1}\ar@{->>}[d]^{\widetilde{\theta}_{4n-1}}\ar@{->>}[r]^{\tilde\eta}&{\sf D}_{4n-1}\ar[d]^{\cong}\\
{\sf K}^{\sf Y}_{4n-1}\ar@{>->}[r]&{\sf Y}_{4n-1}\ar@{->>}[r]&{\sf J}_{4n-1}\\
}
$$

We already know from Theorem~\ref{thm:Levine} that the right-hand square commutes when $\widetilde{\theta}_{4n-1}$ and $\widetilde{\mathcal T}_{4n-1}$ are replaced by $\theta_{4n-1}$ and  
$\cT_{4n-1}$, which implies the right-hand square commutes in the present context. The exactness of the top sequence follows  by combining Theorems 5.1 and 6.5(ii) of \cite{CST1}, using that ${\sf L}_{2n+1}={\sf L}'_{2n+1}$. The surjection $\widetilde{\theta}_{4n-1}$ restricts to a surjection $\z\otimes{\sf L}_{2n+1}\twoheadrightarrow {\sf K}^{\sf Y}_{4n-1}$. By Proposition~\ref{prop:ksurjective}, there is also a surjection in the opposite direction. Since these are $\z$-vector spaces, this is sufficient to conclude  that $\z\otimes {\sf L}_{2n+1}\to{\sf K}^{\sf Y}_{4n-1}$ is an isomorphism. By the $5$-lemma, we can also conclude that the middle map is an isomorphism. This completes the analysis of the order $4n-1$ case.

{\bf Proof of Theorem~\ref{thm:improved-filtration}  (order $4n+1$ case)}\\ 
We again have a commutative diagram of exact sequences
$$
\xymatrix{
\z\otimes{\sf L}'_{2n+2}\ar@{->>}[d]\ar@{>->}[r]&\widetilde{\cT}_{4n+1}\ar@{->>}[d]^{\widetilde{\theta}_{4n+1}}\ar@{->>}[r]^{\tilde\eta}&{\sf D}_{4n+1}\ar[d]^{\cong}\\
{\sf K}^{\sf Y}_{4n+1}\ar@{>->}[r]&{\sf Y}_{4n+1}\ar@{->>}[r]&{\sf J}_{4n+1}\\
}
$$
The $4n+1$ case would follow by the same argument as for $4n-1$, except that $\kappa_{4n+1}$ maps to ${\sf L}_{2n+2}\neq {\sf L}'_{2n+2}$, so we can no longer conclude that the left-hand map is an isomorphism. However, from the following commutative diagram it is clear that $\z\otimes{\sf L}_{n+1}$ surjects onto $\Ker(\kappa_{4n+1})$:
$$\xymatrix{
\z\otimes{\sf L}_{n+1}\ar@{>->}[r]&\z\otimes{\sf L}'_{2n+2}\ar@{->>}[d]\ar@{->>}[r]&\z\otimes{\sf L}_{2n+2}\\
&{\sf K}^{\sf Y}_{4n+1}\ar@{->>}[ur]_{\kappa_{4n+1}}&
}$$
The fact that this diagram commutes will be explained in Remark~\ref{rem:commutative-diagram} during the proof of Proposition~\ref{prop:ksurjective} in subsection~\ref{subsec:proof-prop-ksurjective} below. 
$\square$
%================================================================

\subsection{Clasper conventions and notations}\label{subsec:clasper-conventions}
For details on claspers see e.g.~\cite{CT1}. Given a tree clasper with $n$ internal nodes embedded in a homology cylinder, there is an associated element of $\cT_n$ defined by labeling the univalent vertices of the underlying vertex-oriented tree by the homology classes
determined by the clasper leaves, which are canonically identified with elements of $H_1(\Sigma_{g,1};\mathbb Z)$. The \emph{order} of a tree clasper is the order of its associated tree.
All claspers will be assumed to be of order at least $1$, unless otherwise specified. Surgery on a clasper $C\subset M$ is denoted by $M_C$, and for $C'\subset M\setminus C$, we denote by $C'_C$ the image of $C'$ in $M_C$. 

Rooted trees are usually denoted by capital letters, and unrooted trees are usually denoted by $t$. Lower case letters labeling univalent vertices represent homology elements and capital letters represent sub-trees.
Choosing a preferred leaf of a clasper determines an associated rooted tree.

We will need the \emph{zip construction} as formulated in section 4.2 of \cite{CT1}: Given a clasper $C$, one can cut one of its leaves along a framed arc into two framed embedded loops, yielding two daughter claspers $C_1$ and $C_2$, where $C_1$ is exactly $C$ with one loop of the leaf, and $C_2$ has leaves parallel to those of $C_1$, except where it uses the other loop of the leaf, and its edges and vertices lie in a regular neighborhood of $C_1$ avoiding any caps it might have (a \emph{cap} is a framed embedded disk bounded by a leaf). 

%====================================================================

\subsection{Proof of Proposition~\ref{prop:twistedtheta}}\label{subsec:twistedtheta}
Habiro's surjective realization map $\theta_n\colon\cT_n\twoheadrightarrow {\sf Y}_n$ is defined on a tree  $t$ by choosing a clasper $C$ in the identity homology cylinder $\Sigma_{g,1}\times I$ whose associated tree is $t$. One shows this is well-defined by proving that any two such embeddings of a clasper $C$ are related by $(>n)$-equivalences and homology bordisms, and that the IHX and AS relations are realized geometrically (this argument will be presented and extended in this subsection).

To prove Proposition~\ref{prop:twistedtheta} (stating that $\theta_{2n-1}$ descends to $\widetilde{\theta}_{2n-1}\colon\widetilde{\cT}_{2n-1}\twoheadrightarrow {\sf Y}_{2n-1}$) we introduce the notion of a \emph{twisted clasper} (Definition~\ref{def:twisted-clasper} below), which is required to have one leaf that is a $+1$-framed unknot, and show that it actually represents a $(2n-1)$-equivalence up to homology bordism, when it has $n$ internal nodes, i.e.~when its associated ``twisted'' $\iinfty$-tree is of order $n$ (Definition~\ref{def:twisted-tree}). We will extend Habiro's result to show that the homology cylinders resulting from surgery on any two twisted claspers corresponding to the same $\iinfty$-tree are ${2n}$-equivalent up to homology bordism.

\begin{lem}\label{lem:bordism}
Suppose $M$ is a $3$--manifold, possibly with boundary, and $C$ is a tree clasper embedded in $M$ such that a leaf $\ell$ of $C$ bounds an embedded disk, consistent with the framing of $\ell$, which possibly intersects the rest of $C$. Then $M$ is homology bordant to $M_C$. 
\end{lem}
\begin{proof}
 This is similar to Theorem 2 of \cite{L1} and the the proof given there works with very little modification. Surgery on the clasper $C$ is defined as surgery on a corresponding link $L$, and by lemmas 2.5 and 2.6 of \cite{L1}, to show that surgery on $C$ is a homology bordism, it suffices to decompose $L$ as a union $L'\cup L''$ where $L'$ is a trivial $0$-framed link, and the matrix of linking numbers $\lambda(L'_i,L''_j)$ between the components of $L'$ and $L''$ is non-singular. Orient the edges of the clasper so that every leaf but $\ell$ is at the end of the arrow, $\ell$ is the source of an arrow, and no trivalent vertex is a source (as in Lemma~2.6 of \cite{L1}). Let $L''$ be the link components at the end of the arrows, and let $L'$ be the link components at the beginning of arrows. Then the linking matrix between $L'$ and $L''$ is the identity matrix with one row of possibly non-zero off-diagonal elements, so it is invertible.
\end{proof}
It follows from Lemma~\ref{lem:bordism} that the relation of $n$-equivalence on $\HC(g,1)$ is generated by clasper surgery on \emph{tree} claspers of order (at least) $n$, as observed in \cite{L1}. The relationship between the choice of embedding of a clasper representing a tree-type and the resulting surgery is described by the following lemmas.

\begin{lem}\label{lem:crossing}
Let $C_1\cup C_2\subset M$ be an embedding of two tree claspers, of orders $n_1$ and $n_2$ respectively, in a homology cylinder $M$. Suppose $C_1'\cup C_2'\subset M$ is an embedding that differs from $C_1\cup C_2\subset M$ by:
\begin{enumerate}
\item Crossing a leaf of $C_1$ through a leaf of $C_2$. Then $M_{C_1\cup C_2}\sim_{n_1+n_2}M_{C_1'\cup C_2'}$.
\item Crossing a leaf of $C_1$ through an edge of $C_2$. Then $M_{C_1\cup C_2}\sim_{n_1+n_2+1}M_{C_1'\cup C_2'}$.
\item Crossing an edge of $C_1$ through an edge of $C_2$. Then $M_{C_1\cup C_2}\sim_{n_1+n_2+2}M_{C_1'\cup C_2'}$.
\end{enumerate}
\end{lem}
\begin{proof}
This is standard clasper calculus. It follows by the zip construction and Figure 31 of \cite{Hab}.
\end{proof}

\begin{lem}\label{lem:claspercrossing}
Suppose $C_1$ and $C_2$ are order $n$ tree claspers embedded in a homology cylinder $M$, such that
 $C_2$ is formed from $C_1$ by
 a finger move that pushes an arc of a leaf of $C_1$ across a leaf of $C_1$.
   Then $M_{C_1}\sim_{2n}M_{C_2}$ up to homology bordism.
\end{lem}
\begin{center}
\includegraphics[width=1.5in]{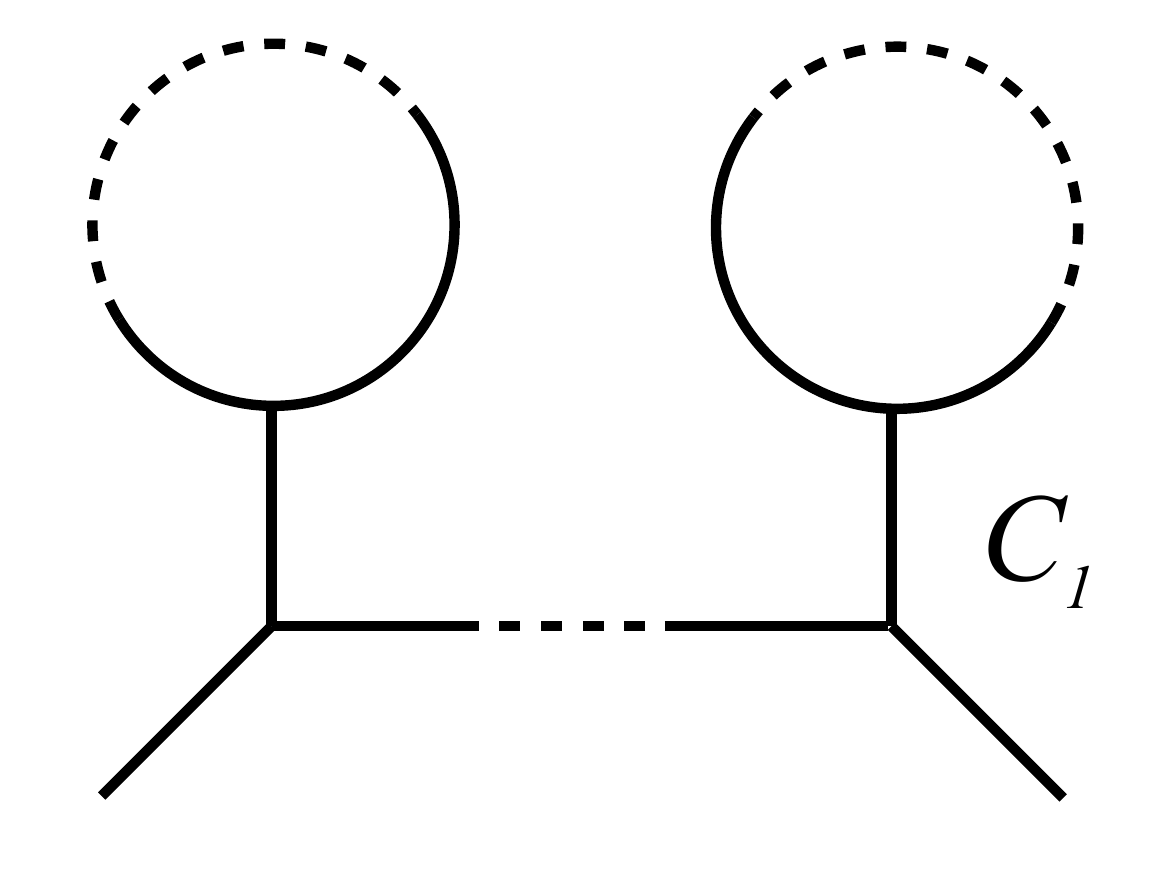}\hspace{2em}\includegraphics[width=1.5in]{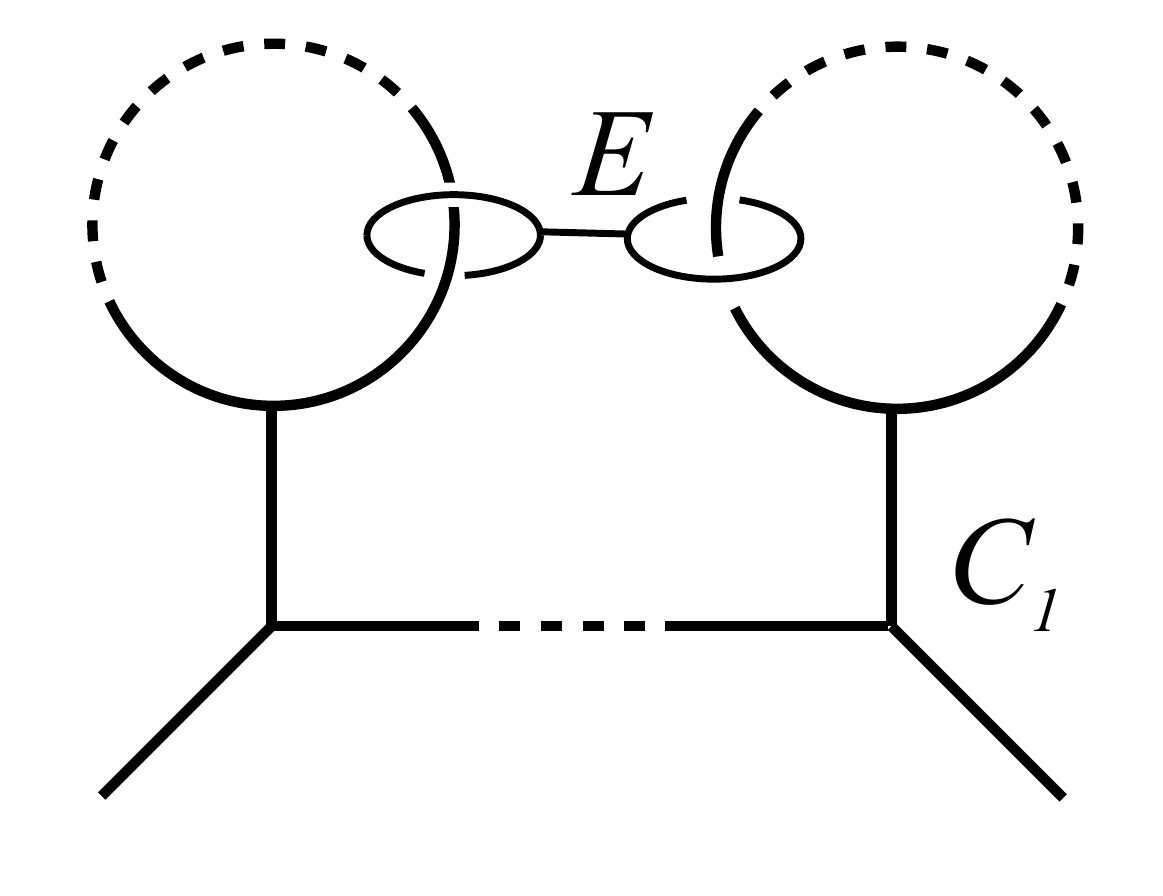}
\end{center}
\begin{proof}
Assume that the crossing is between two different leaves. (The case when it is a self-crossing of a leaf is similar.) As pictured above,
the finger move can be realized by surgery on a clasper $E$, with one edge and two leaves, where each leaf of $E$ is a meridian to one of the two leaves of $C_1$ involved in the finger move. In other words, $C_2=(C_1)_E$.
Now use the zip construction of \cite{CT1} to cut one of the leaves of $C_1$ so that one piece forms a meridian to the $E$ leaf. This gives us two claspers: the original $C_1$ together with $C_1'$, where $C_1'$ has edges that wind through a regular neighborhood $\nu(C_1)$ of $C_1$ in a complicated way, as in the following picture: 
\begin{center}
\includegraphics[width=1.9in]{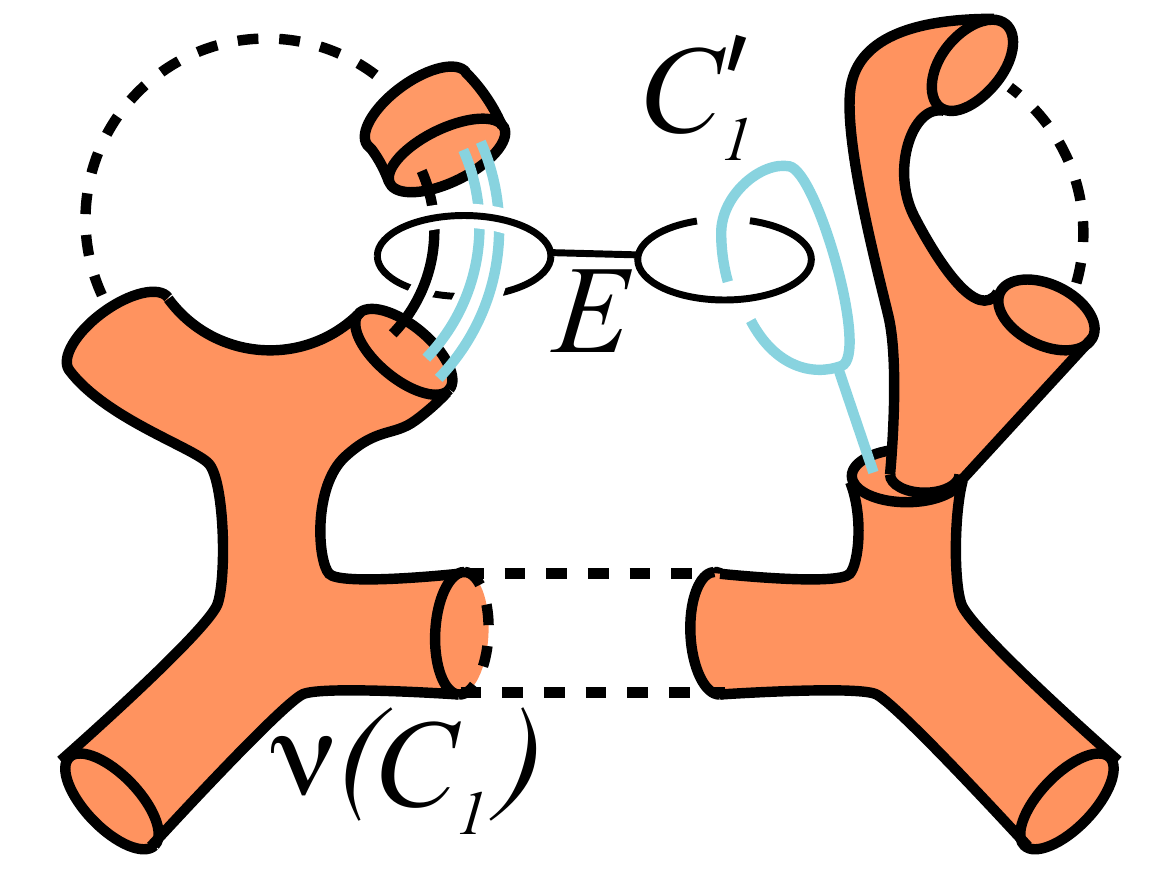}\end{center}
In the preceding picture, $C_1'$ is depicted in blue and the regular neighborhood of $C_1$ is depicted in orange. 
 Near the left-hand $E$-leaf, we have one (black) strand of a leaf of $C_1$, and several parallel (blue) strands which represent a leaf and possibly many edges of $C_1'$.
 
On the left-hand side of the following picture we depict what happens after surgering along the $E$-clasper:
$$
\begin{minipage}{2.0in}
\includegraphics[width=2.0in]{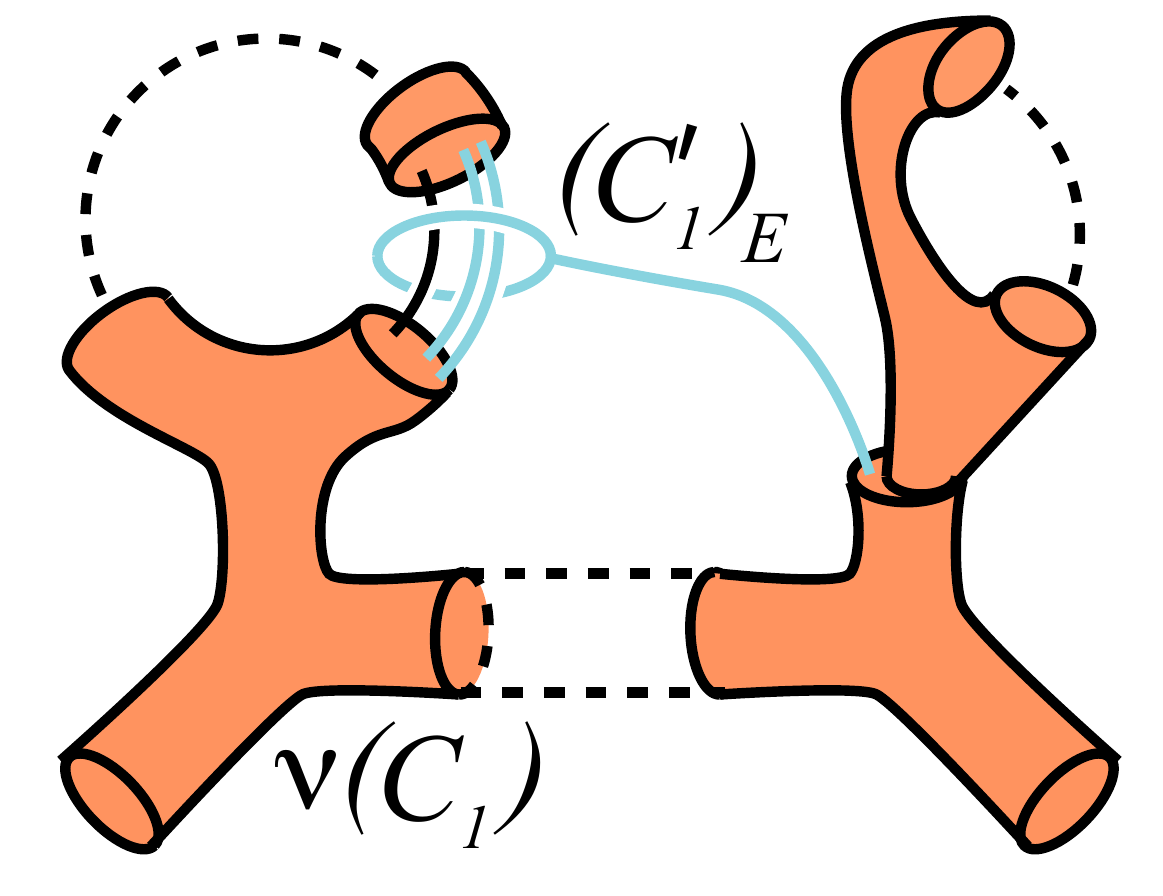}
\end{minipage}
\mapsto
\begin{minipage}{2.0in}
\includegraphics[width=2.0in]{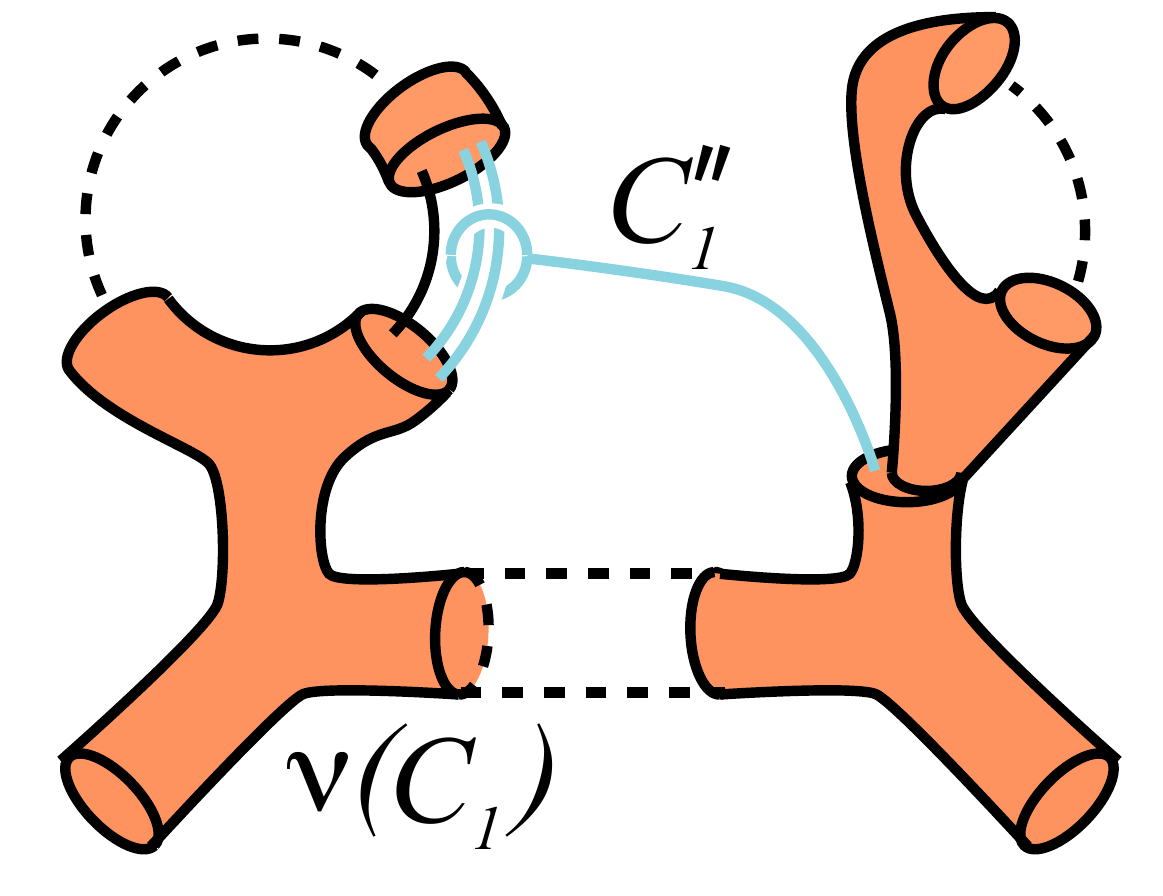}
\end{minipage}
$$
The above-pictured operation of pulling the leaf of $(C_1')_E$ across 
 the indicated leaf of $C_1$ corresponds to a $2n$-equivalence, by Lemma~\ref{lem:crossing}. And by Lemma~\ref{lem:bordism}, clasper surgery on $C_1\cup C_1''$ in the above picture on the right is homology bordant to surgery on $C_1$, completing the proof.
\end{proof}

\begin{lem}\label{lem:edgeleafcrossing}
Let $C_1$ and $C_2$ be order $n$ tree claspers embedded in a homology cylinder $M$ such that $C_2$ differs from $C_1$ by a finger move of an edge of $C_1$ through a leaf of $C_1$, or by a finger move of an edge of $C_1$ though an edge of $C_1$. Then $M_{C_1}\sim_{2n}M_{C_2}$ up to homology bordism.
\end{lem}
\begin{proof}
Any clasper homotopies of this form can be realized by leaf-leaf crossings, which are covered by Lemma~\ref{lem:claspercrossing}. To see this note that a tree clasper is the union of a tree and some circles (the leaves). Fixing a handlebody neighborhood of the clasper,
the tree can be contracted by an isotopy to be in an arbitrarily small ball, pulling parts of the leaves toward this ball, and fixing the complement of the handlebody. The homotopies in the lemma are realized by crossing the handlebody through itself. Since all of the edges are in a small ball, this translates into multiple crossings of the leaves.
\end{proof}

In analogy to twisted Whitney towers we introduce twisted claspers. Notice that in a homology cylinder $M$, every (simple) closed curve has a canonical framing: the curve is homologous into $\partial M$, and this homology determines a parallel push-off of the curve. 
%[IS THIS A GOOD DEFINITION OF CANONICAL FRAMING?]Yes, explain?...

\begin{defn}\label{def:twisted-clasper}
A \emph{twisted clasper of order $2n-1$} is an order $n$ tree clasper $C$ embedded in a homology cylinder 
$M$, such that one leaf of $C$ is a $+1$-framed unknot bounding a disk whose interior is disjoint from $C$, and all other leaves have the canonical framing.
\end{defn}
Here ``$+1$-framed'' means that the canonical parallel push-off of the unknotted leaf has $+1$ intersection with the disk bounded by the leaf.

\begin{defn}\label{def:twisted-tree}
The $\iinfty$-tree associated to a twisted order $2n-1$ clasper $C$ is gotten from the usual order $n$ tree by labeling the univalent vertex which corresponds to the $+1$-framed leaf of $C$ by the symbol $\iinfty$ (this $\iinfty$ label replaces the trivial homology element that would usually label such a vertex). Such a tree is called an $\iinfty$-tree, and 
for any rooted tree $J$ the notation $J^\iinfty$ denotes the $\iinfty$-tree gotten by labeling the root of $J$ by $\iinfty$.
\end{defn}

In the next lemma we see that defining the order of a twisted clasper in this way is justified, as surgery along it
is an equivalence of the correct order in $\HC(g,1)$.

\begin{lem}\label{lem:rightorder}
Let $C$ be a twisted clasper of order $2n-1$ embedded in a homology cylinder $M$. Then $M\sim_{2n-1}M_C$ up to homology bordism.
\end{lem}

\begin{proof}
By geometric IHX relations \cite{CST,CT2} (which preserve framings) one can write $M_C=M_{C_1\cup\cdots\cup C_k}$, where each $C_i$
is a clasper of order $n$ with a $+1$-framed unknotted leaf adjacent to a trivalent vertex which is adjacent to another leaf. The $+1$-framed unknotted leaf of each $C_i$ bounds an embedded disk which intersects the other claspers and may even have interior intersections with $C_i$. 
Using Lemma~\ref{lem:crossing}, the intersections between this disk and the other claspers may be removed, up to ${2n}$-equivalence. Using Lemma~\ref{lem:claspercrossing} and Lemma~\ref{lem:edgeleafcrossing}, interior intersections of this disk with $C_i$ can also be removed, up to ${2n}$-equivalence and homology bordism.
 
 So it suffices to show the lemma is true for any twisted clasper where the $+1$-framed unknotted leaf is next to another leaf.  
By inserting a Hopf link into an edge (Figure~14 of \cite{CT1}), any such clasper of the form { $ \iinfty-\!\!\!<^h_J$} can be converted (without changing the result of surgery) into an order one  $Y$-clasper with one twisted leaf, one leaf of homology class $h$, and one leaf linking a clasper $C_J$ with rooted tree $J$, as  depicted in Figure~\ref{fig:rightorder}A. 
%\centerline{\includegraphics[width=2in]{right-order-clasper-1.pdf}}
Sliding the twisted leaf in the $Y$-clasper over the leaf linking $C_J$ results in two parallel leaves linking the $C_J$ as depicted in Figure~\ref{fig:rightorder}B.
\begin{figure}
\begin{center}
\subfloat[]{\includegraphics[width=2.0in]{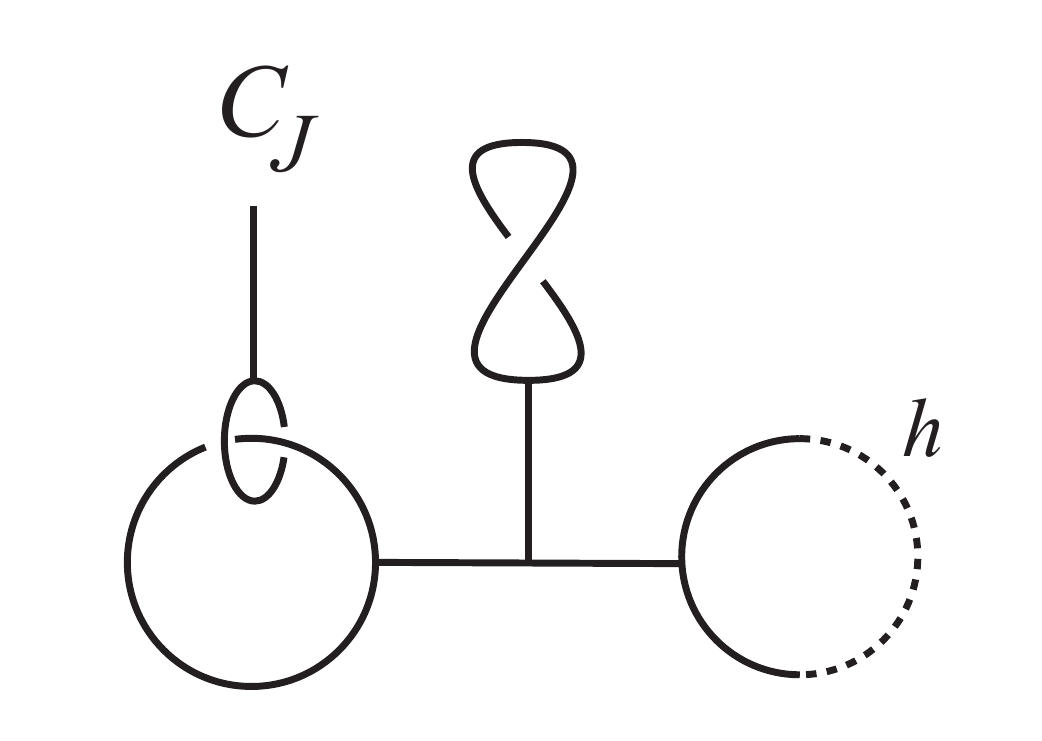}}\hspace{1em}
\subfloat[]{\includegraphics[width=1.85in]{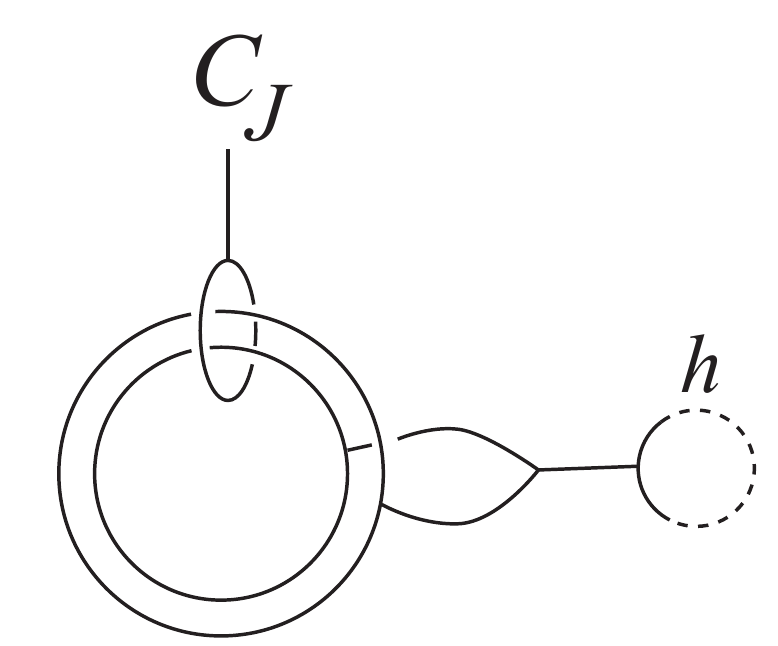}} \hspace{1em}
\subfloat[]{\includegraphics[width=2.0in]{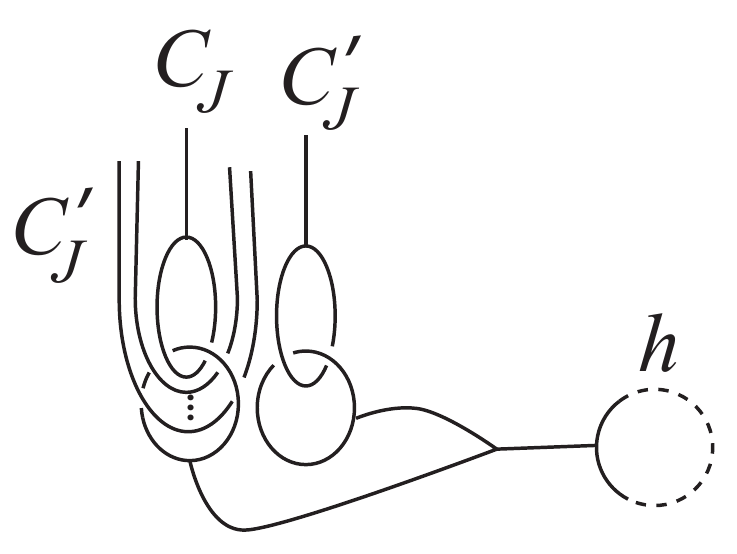}}  \\
\subfloat[]{\includegraphics[width=2.0in]{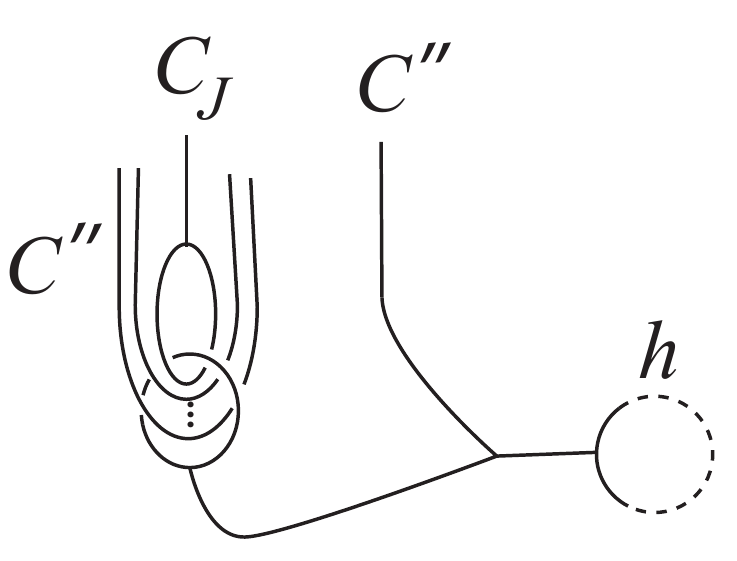}}\hspace{1em}
\subfloat[]{\includegraphics[width=2.0in]{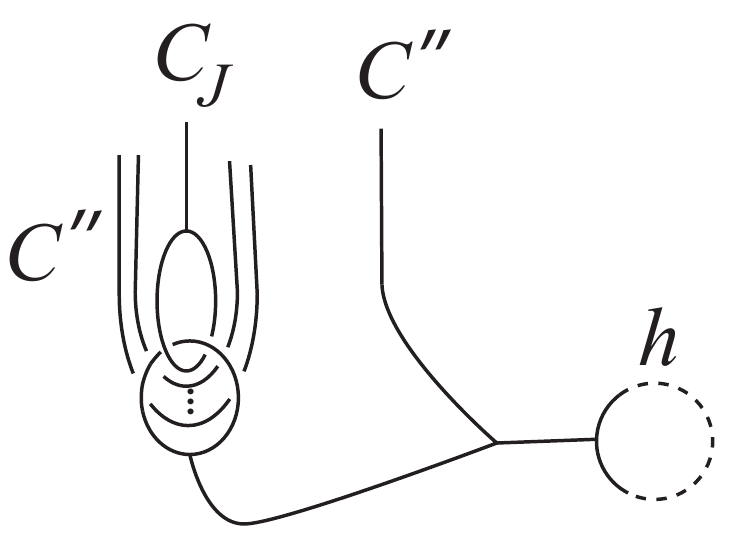}}  \hspace{1em}
\subfloat[]{\includegraphics[width=2.0in]{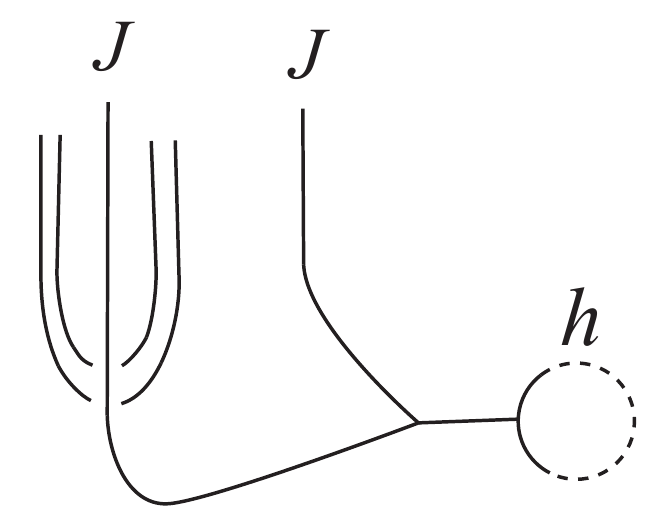}} 
\end{center}
\caption{From the proof of Lemma~\ref{lem:rightorder}.}~\label{fig:rightorder}
\end{figure}
%\\
%\centerline{\includegraphics[width=1.85in]{right-order-clasper-2.pdf}} 
This twisted-leaf-sliding move is pictured in the proof of Lemma 4.9 of \cite{GGP}, and in the proof of Theorem~3.1 of \cite{GGP} it is explained why the twisting gets killed (this is a $3$-dimensional analog of the $4$-dimensional boundary-twist maneuver on a Whitney disk).   Next apply a zip construction \cite{CT1,Hab} to the leaf of $C_J$, splitting apart the parallel leaves of the $Y$-clasper, and turning $C_J$ into claspers $C_J$ and $C'_J$ with parallel leaves and the same tree-type. The edges of $C'_J$ are contained in a neighborhood of $C_J$ (Figure~\ref{fig:rightorder}C).
%\\
%\centerline{\includegraphics[width=2.2in]{right-order-clasper-3.pdf}} \\
Then combine $C'_J$ and the $Y$-clasper by converting the right-hand Hopf link into an edge
of an order $n$ clasper $C''$. The left-hand leaf of $C''$ links the root leaf of $C_J$ and also has strands of $C''$ passing through it (Figure~\ref{fig:rightorder}D).
%\\
%\centerline{\includegraphics[width=2.2in]{right-order-clasper-4.pdf}} \\
By Lemma~\ref{lem:edgeleafcrossing} these strands of $C''$ can be pushed out by an ${2n}$-equivalence up to homology cobordism, so that the leaf forms a clean Hopf link pair with the root leaf of $C_J$ (Figure~\ref{fig:rightorder}E).
%\\
%\centerline{\includegraphics[width=2.2in]{right-order-clasper-5.pdf}} \\
Finally, combining $C_J$ and $C''$ by replacing their Hopf-linked leaves 
with an edge forms an order $2n-1$ clasper of the form  
{ $ \text{\normalsize $h$}-\!\!\!<^J_J$}, which represents an ${2n-1}$-equivalence (Figure~\ref{fig:rightorder}F).
%\\
%\centerline{\includegraphics[width=2.2in]{right-order-clasper-6.pdf}} \\
\end{proof}

\begin{lem}\label{lem:bandsum}
Let $C_1$ be a clasper in a homology cylinder $M$ which is either of order $2n-1$ or is twisted of order $2n-1$. Suppose $C_2$ differs from $C_1$ by band-summing an edge $e$ into a simple closed curve $\gamma$ in $M$. Then $M_{C_1}\sim_{2n}M_{C_2}$, up to homology bordism.
\end{lem}
\begin{proof}
This move is realized by surgery on an order $0$ clasper one of whose leaves is a meridian to the edge, and the other is a copy of $\gamma$, as depicted in Figure~\ref{fig:bandsum}A.
%\centerline{\includegraphics[width=2.5in]{bandsummove.pdf}} 
By clasper calculus (see \cite{Hab}, Figure 30), we can write $M_{C_2}=M_{C_1\cup C'}$, where $C'$ has order $n+1$ and the tree of
$C'$ is gotten from that of $C_1$ by attaching a new edge to an interior point in $e$ (with the leaf of the new edge given by $\gamma$). See Figure~\ref{fig:bandsum}B.
\begin{figure}
\subfloat[]{\includegraphics[width=2.5in]{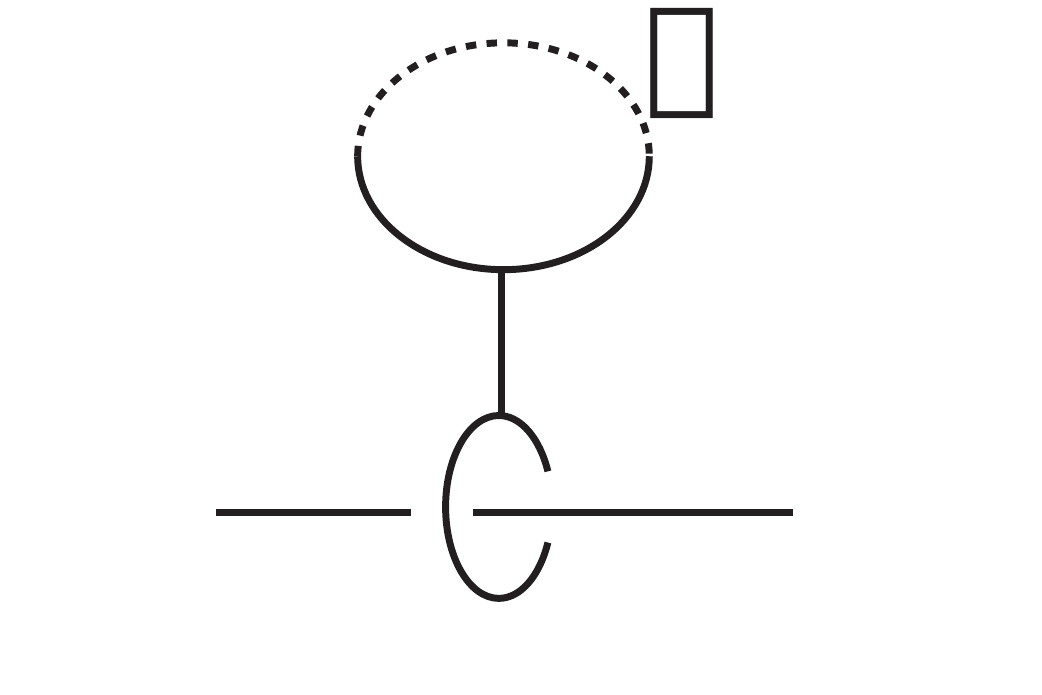}}\hspace{2em}
\subfloat[]{\includegraphics[width=2.5in]{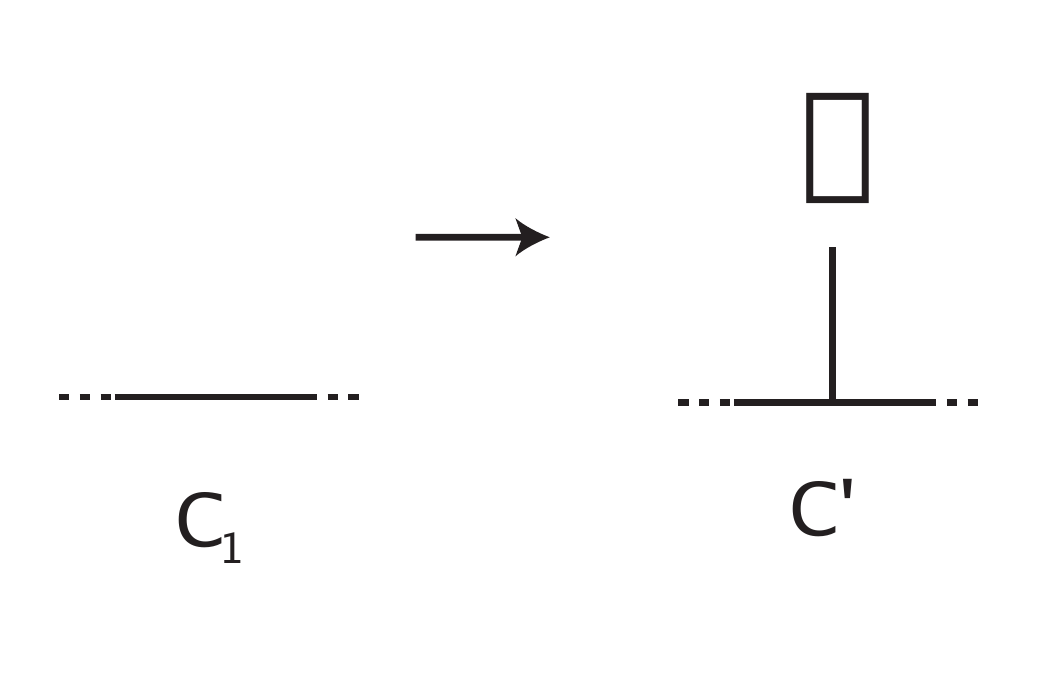}}
\caption{From the proof of Lemma~\ref{lem:bandsum}.}\label{fig:bandsum}
\end{figure}
We may arrange that $C_1$ sits in a regular neighborhood of $C'$, which is the clean daughter clasper in the zip construction. In the case that $C_1$ is of order $2n-1$, we are done, since $C'$ is of order $2n$, so now assume that $C_1$ is twisted of order $2n-1$. A leaf and edges of $C'$ link the twisted leaf of $C_1$  nontrivially, but by Lemma~\ref{lem:crossing}  we can ``disentangle'' the twisted leaf of $C'$ from the clasper $C_1$ by an ${2n+1}$-equivalence so that $C'$ has a clean twisted leaf.  Thus $C'$ has been modified to be twisted of order $2n+1$. So modulo $\sim_{2n}$ we can disregard $C'$, since it determines an ${2n+1}$-equivalence by Lemma~\ref{lem:rightorder}, and we have $M_{C_2}\sim_{2n}M_{C_1}$.
\end{proof}

The following is a version of Lemma~\ref{lem:crossing} for twisted claspers.
\begin{lem}\label{lem:twistedcrossing}
Let $C_1\cup C_2$ be an embedding of two twisted tree claspers in a homology cylinder $M$ each of order $2n-1$. Suppose $C_1'\cup C_2'$ is an embedding that differs by crossing a strand of a leaf of $C_1$ across a leaf of $C_2$. Then $M_{C_1\cup C_2}\sim_{2n}M_{C_1'\cup C_2'}$.
\end{lem}
\begin{proof}
Each of the twisted claspers $C_1$ and $C_2$ can be thought of as a clasper of order $n$. Then Lemma~\ref{lem:crossing} proves that $M_{C_1\cup C_2}\sim_{2n}M_{C_1'\cup C_2'}$.
\end{proof}

We also need a version of Lemma~\ref{lem:claspercrossing} for twisted claspers.

\begin{lem}\label{lem:twistedclaspercrossing}
Let $C_1$ be a twisted tree clasper of order $2n-1$ embedded in a homology cylinder $M$.
If $C_2$ is any twisted tree clasper formed from $C_1$ by a finger move that pushes an arc of a leaf of $C_1$ across another leaf of $C_1$, then $M_{C_1}\sim_{2n}M_{C_2}$ up to homology bordism. 
\end{lem}
\begin{proof}
We trace through the proof of Lemma~\ref{lem:claspercrossing}. The only step of the proof that was not a diffeomorphism or a homology bordism is the step at the end of the proof where a leaf of $(C_1')_E$ is pushed across a leaf of $C_1$, which is a ${2n}$-equivalence  by Lemma~\ref{lem:twistedcrossing}.
\end{proof}

\begin{lem}
Let $C_1$ and $C_2$ be  twisted order $2n-1$ tree claspers embedded in a homology cylinder $M$ which differ by a finger move of an edge through a leaf of $C_1$ or by a finger move of an edge though an edge of $C_1$. Then $M_{C_1}\sim_{2n}M_{C_2}$ up to homology bordism.
\end{lem}
\begin{proof}
As in the proof of Lemma~\ref{lem:edgeleafcrossing}, this follows from the leaf-leaf case.
\end{proof}

\begin{lem}\label{lem:nullhomologous}
Let $C_1$ be a clasper in a homology cylinder $M$ which is either of order $2n-1$ or is twisted of order $2n-1$. Suppose $C_2$ differs from $C_1$ by band-summing a leaf into a null-homologous simple closed curve in $M$. Then $M_{C_1}\sim_{2n}M_{C_2}$ up to homology bordism. 
\end{lem}
\begin{proof}
The curve we are banding into is null-homologous and so it bounds an immersed surface $S$, which may intersect $C_2$. This surface has a $1$-dimensional spine, and by contracting $S$ toward the spine we may take it to be embedded and disjoint from  the rest of $C_2$. This contraction homotopy will introduce crossings of the part of the leaf of $C_2$ lying in $\partial S$ with the rest of $C_2$ and with itself. These are ${2n}$-equivalences up to homology bordism, by Lemmas~\ref{lem:crossing}, \ref{lem:claspercrossing}, \ref{lem:twistedcrossing} and \ref{lem:twistedclaspercrossing}. A zip construction splitting the leaf 
 yields the original clasper $C_1$ together with a clasper $C'$, contained in a neighborhood of $C_1$ except for one leaf that bounds the embedded surface $S$. If $C_1$ is of order $2n-1$ then $S$ can be used to turn $C'$  into a clasper of order $2n$
(using the grope-to-clasper construction). On the other hand if $C_1$ is twisted of order $2n-1$, the twisted leaves of $C_1$ and $C'$ will interlink, but can be separated by a ${2n}$-equivalence up to homology bordism by Lemma~\ref{lem:crossing}. Furthermore any edges of $C_1$ or $C'$ that link with the twisted leaf of $C'$ can be pushed out of that leaf by Lemmas~\ref{lem:crossing} and ~\ref{lem:claspercrossing}, so that $C'$ becomes twisted of order $2n+1$, so can be ignored as it represents a ${2n+1}$-equivalence up to homology bordism.
\end{proof}

%\begin{lem}\label{lem:framing}
%Let $C$ be a tree clasper with $n$ trivalent vertices embedded in a homology cylinder $M$, where one leaf is $+2$ framed unknot in a copy of $B^3$. Then $M_C\sim_{2n}M$.
%\end{lem}
%\begin{proof}
%Using the zip construction to cut the unknot into two $+1$-framed unknots, we get two twisted daughter claspers $C_1',C_1''$ of order $2n-1$. Then we claim that $M_{C_1}\cdot M_{C_1}\sim_{2n}1$. In the proof of Lemma~\ref{lem:rightorder} we saw that surgery on $C_1$ is equivalent to surgery on a union of claspers of order $2n-1$ represented by symmetric trees of the form $\langle i,(J,J)\rangle.$ Thus it suffices to show surgery on a clasper of order $2n-1$ with this symmetric tree type gives a $2$ torsion element of $\mathcal H_g/\sim_{2n}$. But now we can use that Habiro has shown [ref] that the realization map  $\mathcal T_{2n-1}\to \mathcal H_g/\sim_{2n}$ is well-defined, so that the $2$ torsion of the tree is geometrically realized.
%\end{proof}

%There is a map from claspers of order $2n-1$ to trees of order $2n-1$given by forgetting the embedding and labeling the leaves of the tree by the homology classes of the corresponding embeddings.
%There is a similar map from twisted claspers to $\iinfty$-trees where the twisted leaf gets marked with an $\iinfty$ symbol. 

\begin{lem}\label{lem:independent}
If $C$ and $C'$ are two (twisted) claspers of order $2n-1$ embedded in a homology cylinder $M$, such that $C$ and $C'$ have the same tree type,
then $M_C\sim_{2n}M_{C'}$ up to homology bordism.
\end{lem}
\begin{proof}
First consider the untwisted case.
Two such embeddings $C\subset M$ and $C'\subset M$ differ, up to framing, by a sequence of crossings, band-sums of edges into nontrivial curves and band-sums of leaves into null-homologous curves, all of which are ${2n}$-equivalences and homology bordisms by Lemmas~\ref{lem:claspercrossing}, \ref{lem:edgeleafcrossing}, \ref{lem:bandsum} and \ref{lem:nullhomologous}. By introducing crossings of a leaf with itself (Lemma~\ref{lem:claspercrossing}), one can change the framing by multiples of $2$. We can change the leaf framings in this way to be either $0$ or $+1$. In the latter case, 
by the zip construction, one can split the leaf into a $0$-framed leaf and a $+1$-framed unknot. Any clasper of (untwisted) order $2n-1$ with a $+1$-framed unknotted leaf represents a higher-order (${4n-3}$)-equivalence up to homology cobordism by Lemma~\ref{lem:rightorder} (which also holds if more than one leaf is twisted). Thus we may assume all leaves are canonically framed. 

In the twisted case, all framings on the leaves of $C$ and $C'$ are assumed to agree, and we can again move the two claspers into the same position by a sequence of crossings, band-sums of edges into nontrivial curves and band-sums of leaves into null-homologous curves, all of which are ${2n}$-equivalences up to homology bordism.
\end{proof}

\begin{defn}[\cite{CST1}]\label{def:reduced-tree-group}
The group $\widetilde{\mathcal T}_{2n-1}$ is defined as ${\mathcal T}_{2n-1}/\im\Delta_{2n-1}$, where the homomorphism $\Delta_{2n-1}\colon\z\otimes \mathcal{T}_{n-1}\to\mathcal{T}_{2n-1}$ is defined as follows.
Suppose $t$ is a generator of $\mathcal T_{n-1}$. For any univalent vertex $v$ of $t$, denote by $\ell(v)$ the label of that vertex, and denote by $T_v(t)$ the rooted tree obtained from $t$ by thinking of $v$ as the (unlabeled) root. With these notations 
$$\displaystyle\text{\normalsize $\Delta_{2n-1}(t)=$}\sum_v\,\text{\normalsize $\ell(v) $}-\!\!\!<^{T_v(t)}_{T_v(t)}.$$
\end{defn}

The following presentation of $\widetilde{\mathcal T}_{2n-1}$ will be used to complete the proof of Proposition~\ref{prop:twistedtheta}:
\begin{lem}\label{lem:reduced-Tau-groups-with-twisted-trees}
The group $\widetilde{\mathcal T}_{2n-1}$ is isomorphic to the abelian group generated 
by order $2n-1$ (non-$\iinfty$) trees and order $n$ $\iinfty$-trees, with the following relations:
\begin{enumerate}
\item  IHX and AS relations among order $2n-1$ trees.
%\item  $2J^\iinfty=0$, for every $\iinfty$-tree $J^\iinfty$.
\item IHX relations among $\iinfty$-trees.
\item Boundary twist relations: $ \text{\normalsize $h$}-\!\!\!<^\iinfty_J\,=\text{\normalsize $h$}-\!\!\!<^J_J$, where $h$ is a tree of order $0$.
\end{enumerate}
\end{lem}
%The relations involving $\iinfty$-trees correspond to geometric moves involving twisted Whitney disks, namely the interior twist, IHX, and boundary twist constructions \cite{CST1}, all of which have analogues in the $3$-dimensional clasper setting as illustrated by the results of this section.

\begin{proof}
Let $\overline{\mathcal T}_{2n-1}$ be the group defined by the above presentation. Define a map $\widetilde{\cT}_{2n-1}\to\overline{\cT}_{2n-1}$ by sending each order $2n-1$ generator to itself.  Clearly this vanishes on IHX and AS relators, since these are also present in the target group. A framing relator $\Delta(t)$ is sent to the sum of attaching an $\iinfty$-labeled edge to the edge adjacent to every univalent vertex of $t$, using the boundary twist relations in the target group. This sum is zero as a consequence of IHX relations on $\iinfty$-trees, so the map is indeed well-defined. On the other hand, we define a map $\overline{\cT}_{2n-1}\to\widetilde{\cT}_{2n-1}$ back the other way as follows. Send each order $2n-1$ generator to itself, and send each $\iinfty$-tree,  { $ \iinfty-\!\!\!<^J_K$}, to the sum:
$$\displaystyle\sum_{v\in J}\,\text{\normalsize $\ell(v) $}-\!\!\!<^{T_v(J-\!-K)}_{T_v(J-\!-K)}$$
where we sum over all leaves of $J$ in the tree $J-\!\!\!-\,K$. This is well-defined, since doing this same construction with $K$ differs by $\Delta(J-\!\!\!-\,K)$ from doing it with $J$. This map clearly vanishes on IHX and AS relators for order $2n-1$ trees, as well as $2J^\iinfty$ relators. To see that it vanishes on IHX relators for $\iinfty$-trees, notice that these are generated by IHX relators on $\iinfty$-trees where the $\iinfty$ is adjacent to a leaf:
{ $ \tree{\iinfty}{I}{k}-\tree{\iinfty}{H}{k}+\tree{\iinfty}{X}{k}$}. This maps to $\sum_v \tree{k}{T_v(I-H+X)}{T_v(I-H+X)}$, which vanishes by IHX relations for non-$\iinfty$ trees.
The IHX relators involving the $\iinfty$-edge also map to $0$, as do the relators corresponding to the boundary twist relations. 
\end{proof}

Define the map $\widetilde\theta_{2n-1}\colon \widetilde{\mathcal T}_{2n-1}\to {\sf Y}_{2n-1}$ on generators by sending each tree to $(\Sigma_{g,1}\times I)_C$, where $C$ is any clasper realizing that tree type, both in the twisted and untwisted cases. 

\begin{lem}\label{prop:well-defined}
The homomorphism $\widetilde{\theta}_{2n-1}\colon \widetilde{\mathcal T}_{2n-1}\to{\sf Y}_{2n-1}$
is well-defined.
\end{lem}
\begin{proof}
The definition of $\widetilde\theta_{2n-1}$ does not depend on the choices of embeddings of claspers by Lemma~\ref{lem:independent}. So it suffices to check that $\widetilde\theta_{2n-1}$ respects the relations in Lemma~\ref{lem:reduced-Tau-groups-with-twisted-trees}. The fact that $\widetilde\theta_{2n-1}$ vanishes on IHX relators follows because all IHX relations can be realized geometrically and the claspers can be pulled apart modulo order $2n$ clasper surgeries, as we mentioned in the proof of Lemma~\ref{lem:rightorder}. The boundary twist relations were also shown in that proof. Antisymmetry relations follow from Figure~45 of \cite{Hab} or Corollary~4.6 of \cite{GGP}.
 This implies that if $C$ is a clasper, then there is a clasper $\widetilde{C}$ in the complement of $C$ of the same tree type, except that a half-twist has been introduced into an edge, such that $M_{C\cup \widetilde{C}}\cong M$. These two framings correspond to opposite trees in $\widetilde{\mathcal T}_{2n-1}$. 

The boundary twist relation was proven explicitely in the argument for Lemma~\ref{lem:rightorder}.
\end{proof}

The proof of Proposition~\ref{prop:twistedtheta} is completed by the observation that the surjectivity of $\widetilde{\theta}_{2n-1}$ stated in Proposition~\ref{prop:twistedtheta}
follows since all relations are realized in target, and the original $\theta$-map was onto.

%====================================================================

\subsection{Proof of Proposition~\ref{prop:ksurjective}}\label{subsec:proof-prop-ksurjective}
Recall that Proposition~\ref{prop:ksurjective} asserts that the homomorphism $\kappa_{2n-1}\colon{\sf K}^{\sf Y}_{2n-1}\to\z\otimes{\sf L}_{n+1}$ is surjective, where for $[M]\in{\sf K}^{\sf Y}_{2n-1}:=\Ker\{{\sf Y}_{2n-1}\to {\sf J}_{2n-1}\}$, $\kappa_{2n-1}([M])$ is defined by mapping $[M]\in{\sf J}_{2n}\cong{\sf D}_{2n}$ into $\z\otimes{\sf L}_{n+1}$ via the right-hand map in the top sequence of Theorem~\ref{thm:filtration}.
This latter map $s\ell_{2n}\colon {\sf D}_{2n}\to\z\otimes{\sf L}_{n+1}$ is defined by applying the snake lemma to the following diagram (see \cite[Defn.~5.6]{CST1}):
$$\xymatrix{
&&\mathbb Z_2\otimes\sL_{n+1}\ar@{>->}[d]^{sq}\\
\sD'_{2n}\ar@{>->}[d] \ar@{>->}[r]&\sL_1\otimes\sL'_{2n+1}\ar@{>->>}[d]^\cong\ar[r]&\sL'_{2n+2}\ar@{->>}[d]\\
\sD_{2n} \ar@{>->}[r]\ar@{-->>}[d]^{s\ell_{2n}}&\sL_1\otimes \sL_{2n+1}\ar@{->>}[r]&\sL_{2n+2}\\
\z\otimes \sL_{n+1}&&
}$$
The two horizontal sequences are exact by definition and the vertical sequence on the right is exact by Theorem 2.2 of \cite{L3}. The squaring map on the upper right is $sq(1\otimes X) := [X,X]$.

Let $J$ be any rooted tree of order $n$. Then $J$ represents an element in $\z\otimes \sL_{n+1}$, and $J^\iinfty$ represents a generator in $\widetilde{\cT}_{2n-1}$.
To prove Proposition~\ref{prop:ksurjective} we will first show that $\widetilde{\theta}_{2n-1}(J^\iinfty)\in{\sf K}^{\sf Y}_{2n-1}$, and then check that 
that $\kappa_{2n-1}(\widetilde{\theta}_{2n-1}(J^\iinfty))=J$. 

First we describe $\widetilde{\theta}_{2n-1}(J^\iinfty)$ geometrically. Consider an embedded twisted clasper $C\subset\Sigma_{g,1}\times I$ which has the tree type of $J^\iinfty$, and assume that all the non-twisted leaves are parallel copies of a symplectic basis $\{x_i,y_i\}_{i=1}^g$ for $\Sigma_{g,1}$, with the twisted leaf being a clean $+1$ framed unknot (so each univalent label on $J^\iinfty$ is an $x_j$ or $y_j$ except for the $\iinfty$-labeled vertex).

Then $\widetilde{\theta}_{2n-1}(J^\iinfty)=[(\Sigma_{g,1}\times I)_C]$, and
we want to geometrically understand the element of $\sD_{2n}$ determined by 
$(\Sigma_{g,1}\times I)_C$.

Recall from \cite{GL,L1,L2} that $\J_k$ is defined to be the kernel of the map $\HC(g,1)\to A_0(F/F_{k+1})$ induced by the composition $(\iota^-)^{-1}\circ\iota^+$, with $F$ the free group on $\{x_i,y_i\}_{i=1}^g$. There is a short exact sequence $0\to {\sf D}_k\to A_0(F/F_{k+2})\to A_0(F/F_{k+1})\to 0$, which implies that $\sJ_k\cong\sD_k$. To read off an element in ${\sf D}_k$ for a given homology cylinder representing an element in $\sJ_k$, consider the corresponding automorphism $f$ of $F/F_{k+2}$ which by hypothesis acts trivially on $F_{k+1}$. For each $x\in F/F_{k+2}$, we have $f(x)=x\psi(x)$, where $\psi(x)\in F_{k+1}/F_{k+2}\cong{\sf L}_{k+1}$. Via the isomorphism $\Hom(H,{\sf L}_{k+1})\cong H\otimes {\sf L}_{k+1}$ with $\{x_i,y_i\}_{i=1}^g$ as a basis for $H:= F/F_2$, $f$ is described by $\sum_i x_i\otimes \psi(y_i)- y_i\otimes\psi(x_i)$ which is an element of ${\sf D}_k$ (see the proof of Prop. 2.5 in \cite{GL}). 

We want to calculate this element in the case of $(\Sigma_{g,1}\times I)_C$. 
$f$ is calculated on $x_i$ by homotoping (a circle representing) $x_i$ from $\Sigma_{g,1}\times  \{0\}$ to $\Sigma_{g,1}\times  \{1\}$, and 
such a homotopy can be arranged so that $x_i$ intersects once with each leaf of $C$ which is labeled by the dual element $y_i$, and misses all other leaves of $C$. 
Each time $x_i$ pushes through a leaf $\ell_{y_i}$ of $C$ the element $\psi(x_i)$ picks up a meridian of the clasper, which by the Lemma~\ref{lem:commutator} and Lemma~\ref{lem:twistedcommutator} below corresponds to the iterated commutator determined by putting a root on the doubled tree $J-\!\!\!-\,J$ at the univalent vertex corresponding to $\ell_{y_i}$ in one of the $J$ sub-trees. So $\psi(x_i)$ is a product of such commutators (which we write as a sum of brackets in the abelian group $F_{2n+1}/F_{2n+2}\cong{\sf L}_{2n+1}$), with the terms of the product corresponding to the iterated commutators determined by putting roots at all the $y_i$-labeled vertices in one $J$-subtree of $J-\!\!\!-\,J$. A similar discussion holds for $\psi(y_i)$ except for a minus sign coming from 
the identification of $H$ with its dual $H^*$ via the symplectic form ($x_i\mapsto -y_i^*$ and
$y_i\mapsto x_i^*$).
The upshot is that the element $\sum_i x_i\otimes \psi(y_i)- y_i\otimes\psi(x_i)$ is equal to
$$
\sum_{v\in J} \ell (v) \otimes (J-\!\!\!-\,J)_v
$$
where the sum is over all univalent vertices $v$ in one $J$-subtree of $J-\!\!\!-\,J$, and $(J-\!\!\!-\,J)_v$ denotes
the element of $\sL_{2n+1}$ determined by putting a root at $v$. This sum is equal to $\frac{1}{2}\eta_{2n}(J-\!\!\!-\,J)$, where $\eta_{2n}\colon\cT_{2n}\to\sD_{2n}$ is the map in equation (1) of \cite{L2}. Since the coefficient of $\eta_{2n}(J-\!\!\!-\,J)$
is even and $\sD_{2n}$ is torsion-free, we conclude that $\sum_i x_i\otimes \psi(y_i)- y_i\otimes\psi(x_i)$ lies in $\sD_{2n}$.

So $[(\Sigma_{g,1}\times I)_C]$ maps to $\frac{1}{2}\eta(J-\!\!\!-\,J)\in{\sf D}_{2n}$, and in particular is an element of ${\sf K}^{\sf Y}_{2n-1}$.

Tracing through the application of the snake lemma in the definition of $s\ell_{2n}$ above shows that 
$s\ell_{2n}(\frac{1}{2}\eta(J-\!\!\!-\,J))=J$, so $\kappa_{2n-1}$ is surjective. 
$\hfill\Box$

\begin{rem}\label{rem:commutative-diagram} Note that the commutativity of the last diagram in subsection~\ref{subsec:proof-thm-improved-filtration} used in the proof of Theorem~\ref{thm:improved-filtration} comes from the fact that if $J\in\z\otimes \sL'_{2n+2}$ and $C$ is a clasper with this tree-type, then the element in $\sD_{4n+2}$ represented by $(\Sigma_{g,1}\times I)_C$ is $\frac{1}{2}\eta_{4n+2}(J-\!\!\!-\,J)$, which maps by $s\ell_{4n+2}$ to $J\in \z\otimes \sL_{2n+2}$ as required.
\end{rem}

Let $C$ be a tree clasper of order $n$ embedded in an oriented $3$--manifold $M$ with fundamental group $G:=\pi_1M$. Orient $C$ as a surface and orient the leaves consistent with this orientation. Consider a curve $\gamma$ which is an oriented positive meridian to a leaf of $C$, connected by a whisker to the basepoint. Connect the cores of each leaf of $C$ by whiskers to the basepoint, so that each leaf represents an element of $G$. Thinking of the leaf which $\gamma$ links as a root, the tree associated to $C$ determines an iterated commutator (bracketing expression) $c\in G_n$ of the group elements of the other leaves. For example, in Figure~\ref{fig:commutator}, $c=[\alpha_1,[\alpha_2,\alpha_3]]$.

\begin{lem}\label{lem:commutator}
$\gamma=(-1)^{n-1}c\in G_n/G_{n+1}$. 
\end{lem}
\begin{proof}
This is proven by induction on $n$. Note that modulo higher order commutators, the choice of whiskers to the basepoint doesn't matter. So assume that the whiskers actually follow the clasper's body. Now form an iterated commutator $\widetilde{c}$ defined like $c$, except that instead of $[u,v]$, we form $[u,v^{-1}]$.
With these specific choices of whisker, we claim that $\gamma=\widetilde{c}$.
The base case is when $n=1$. In this case, it's not hard to see that $\gamma=[\alpha_1,\alpha_2^{-1}]=-[\alpha_1,\alpha_2]\in G/G_2$ as desired (see e.g.~Figure~23 of \cite{CT1}). Given a clasper of order $n$ break it into a clasper $C'$ of  order $n-1$ and a Y-clasper which links a leaf of $C'$, and the other two leaves are the group elements $\alpha$ and $\beta$. Then by the base case, the leaf linking $C'$ is the group element $[\alpha,\beta^{-1}]$, feeding this into the inductive computation of $\widetilde{c}'$, yields the desired $\widetilde{c}$.
\end{proof}

\begin{figure}
\begin{center}
\includegraphics[width=1.5in]{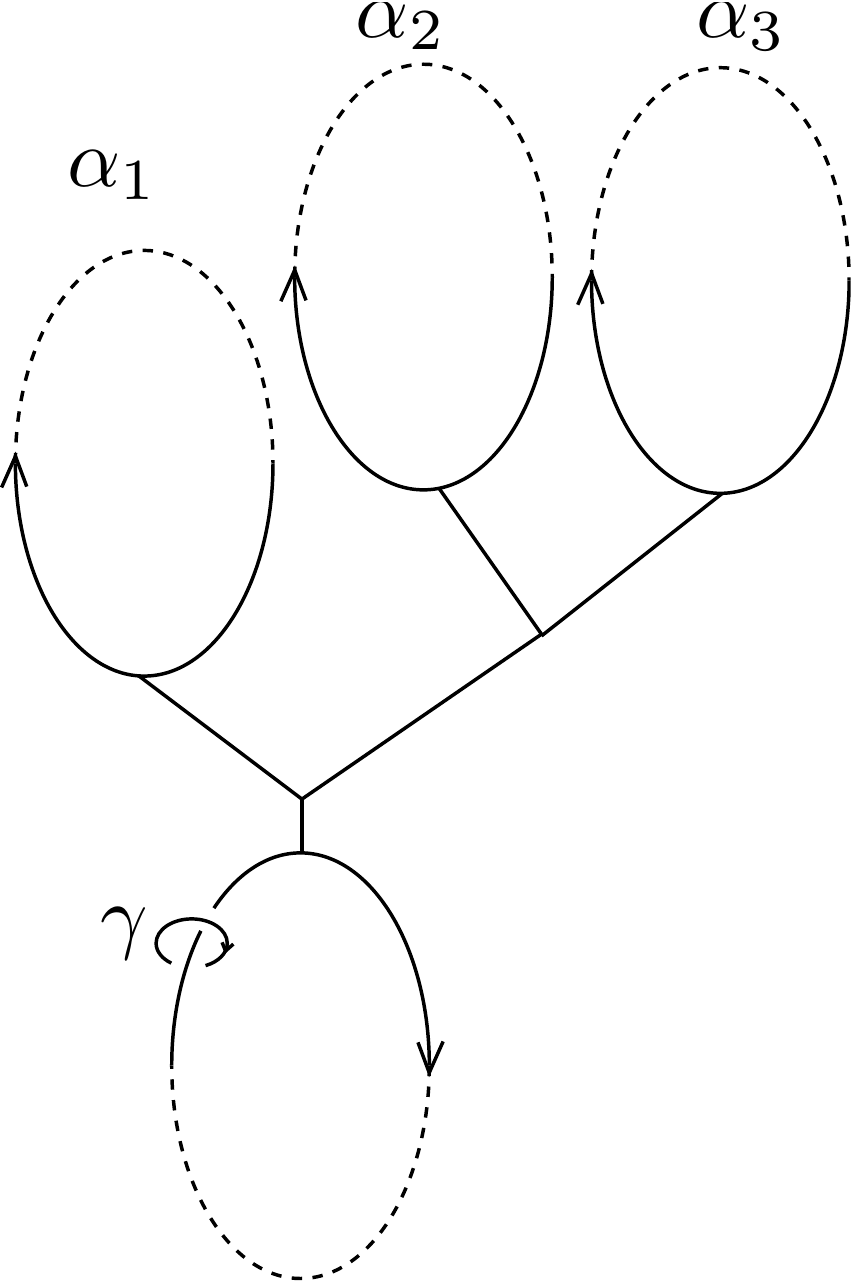}\hspace{1.5in} \includegraphics[width=1.5in]{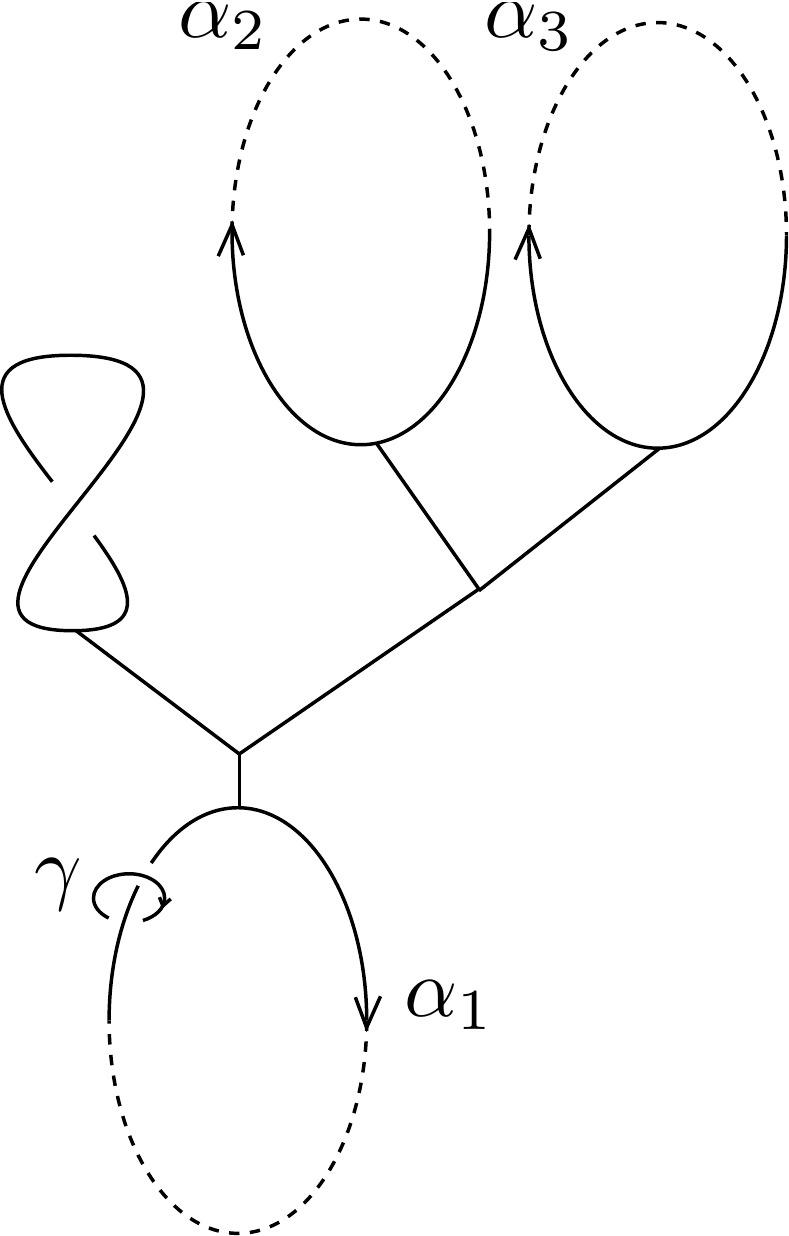}
\end{center}
\caption{Left: An order $2$ tree clasper in a $3$--manifold. According to Lemma~\ref{lem:commutator},  $\gamma=-[\alpha_1,[\alpha_2,\alpha_3]]\in G_3/G_4$. Right: An order $4$ twisted tree clasper. By Lemma~\ref{lem:twistedcommutator}, $\gamma=[[[\alpha_2,\alpha_3],\alpha_1],[\alpha_2,\alpha_3]]\in G_5/G_6$.
 }\label{fig:commutator}\end{figure}

Suppose in the setting of the previous lemma $C$ is a twisted clasper of order $n$, where $\gamma$ does not link the twisted leaf. Thinking of the $\gamma$-leaf as a root, the $\iinfty$-tree associated to $C$ determines a formal commutator $c_\iinfty$, where the twisted leaf corresponds to an $\iinfty$-symbol. Thinking of the $\iinfty$-vertex as the root, we get another iterated commutator of group elements $\omega$. Let $c$ be the commutator obtained by replacing the $\iinfty$-symbol in $c_\iinfty$ by $\omega$. For example, in Figure~\ref{fig:commutator}, $c_\iinfty=[\iinfty,[\alpha_2,\alpha_3]]$, $\omega=[[\alpha_2,\alpha_3],\alpha_1]$,and $c=[[[\alpha_2,\alpha_3],\alpha_1],[\alpha_2,\alpha_3]]$.

\begin{lem}\label{lem:twistedcommutator}
$\gamma=c\in G_{2n-1}/G_{2n}$.
\end{lem}
\begin{proof}
The $\iinfty$-symbol in the commutator $c_\iinfty$ geometrically corresponds to a longitude of the twisted leaf, which is homotopic to a meridian, so by the previous lemma, represents the commutator $\omega$. Furthermore, there are two factors of $(-1)^n$ which cancel each other out. 
\end{proof}

%=======================================================================================================================

\section{On a map from string links to homology cylinders}\label{sec:map}
In this section we use a grope-to-Whitney tower construction from \cite{CST} together with algebraic results from \cite{CST1} and Levine's Theorem~\ref{thm:Levine} mentioned earlier to describe the relationship between the graded groups associated to the Whitney tower filtration $\mathbb {SW}_n$ and the $Y$-filtration $\mathbb Y_n$ via a geometric map from string links to homology cylinders. As we mentioned in the introduction, Habegger defined a bijection between the sets $\HC(0,2g+1)$ and $\HC(g,1)$, where $\HC(0,2g+1)$ can be identified with the set of $(2g+1)$-string links in a homology ball \cite{H}.
Habegger used this bijection to equate the Milnor and Johnson invariants.
By Theorem~\ref{thm:improved-filtration} above and the Milnor invariant-Whitney tower relationship (e.g.~\cite{CST0}), the results of this section can be thought of as a geometric analogue of the Milnor-Johnson correspondence described in \cite{H}.

\begin{figure}
\begin{center}
\begin{eqnarray*}
\begin{minipage}{.6\linewidth}
\includegraphics[width=\linewidth]{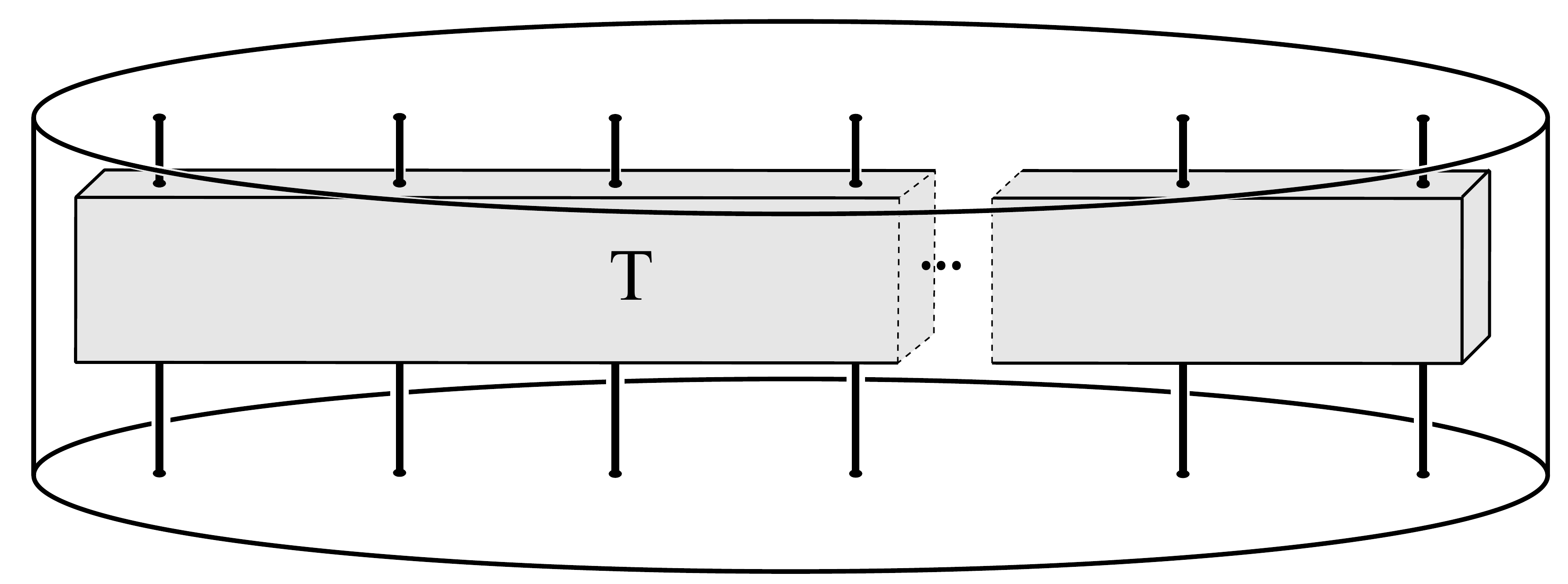} 
\end{minipage}
\overset{H}{\longrightarrow} \text{\hspace{5em}} \\
\begin{minipage}{.6\linewidth}
\includegraphics[width=\linewidth]{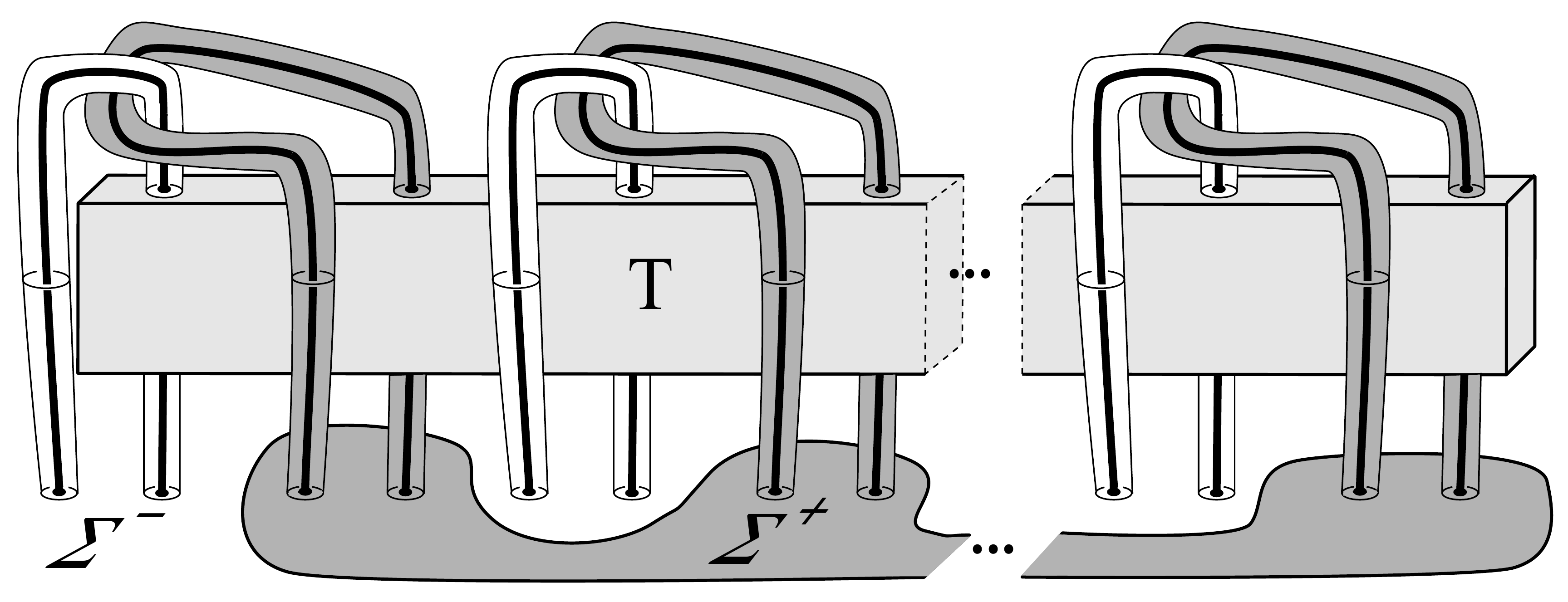}
\end{minipage}
\end{eqnarray*}
\caption{Closing a string link (defined by a tangle $T$) to a homology cylinder.}\label{fig:stringclosure}
\end{center}
\end{figure}

Figure~\ref{fig:stringclosure} illustrates a construction which sends a $2g$ component string link with trivial linking numbers to a homology cylinder over $\Sigma_{g,1}$: The arcs in the right-hand side have endpoints on a piece of a plane, which is the bottom of a topological ball containing the tangle $T$ which defines the string link pictured on the left. We one-point compactify the bottom plane to $S^2$ and think of everything as taking place in a 3--ball $B^3$ bounded by that 2--sphere. Remove from $B^3$ a regular neighborhood of $T$ in the right-hand picture to get a $3$--manifold with boundary that decomposes into two surfaces $\Sigma^{\pm}$ as pictured. Since the pairwise linking numbers of the original string link are zero, the result is a homology cylinder over $\Sigma_{g,1}$. Moreover, the trivial string link gets sent to the identity product $\Sigma_{g,1}\times I$. This construction is not multiplicative on the monoids of string links and homology cylinders, but does give a well-defined map 
\[
\mathrm{H}\colon\mathbb{SL}_1(2g)\to\HC_1(g,1)
\] 
We recall that $\mathbb{SL}_1(2g)$ is the set of concordance classes of string links  with $2g$ strands and trivial linking matrix, and $\HC_1(g,1)$ are homology bordism classes of homology cylinders inducing the identity map on first homology $H_1(\Sigma_{g,1})$. The map $\mathrm{H}$ takes  concordant string links to homology cobordant homology cylinders, and is therefore well-defined.

\begin{figure}
\begin{center}
\includegraphics[width=2.5in]{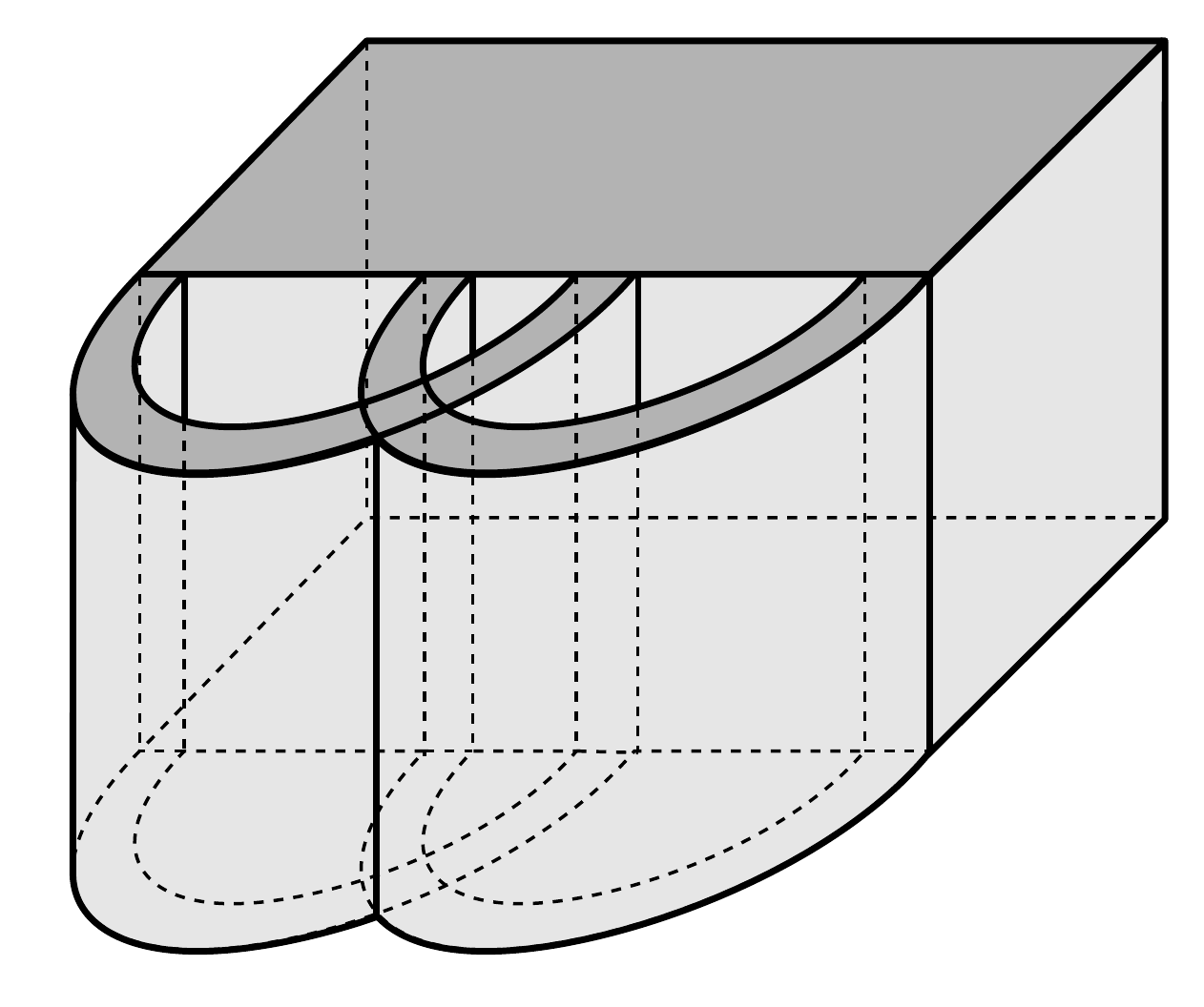}\hspace{.2in}\includegraphics[width=2.5in]{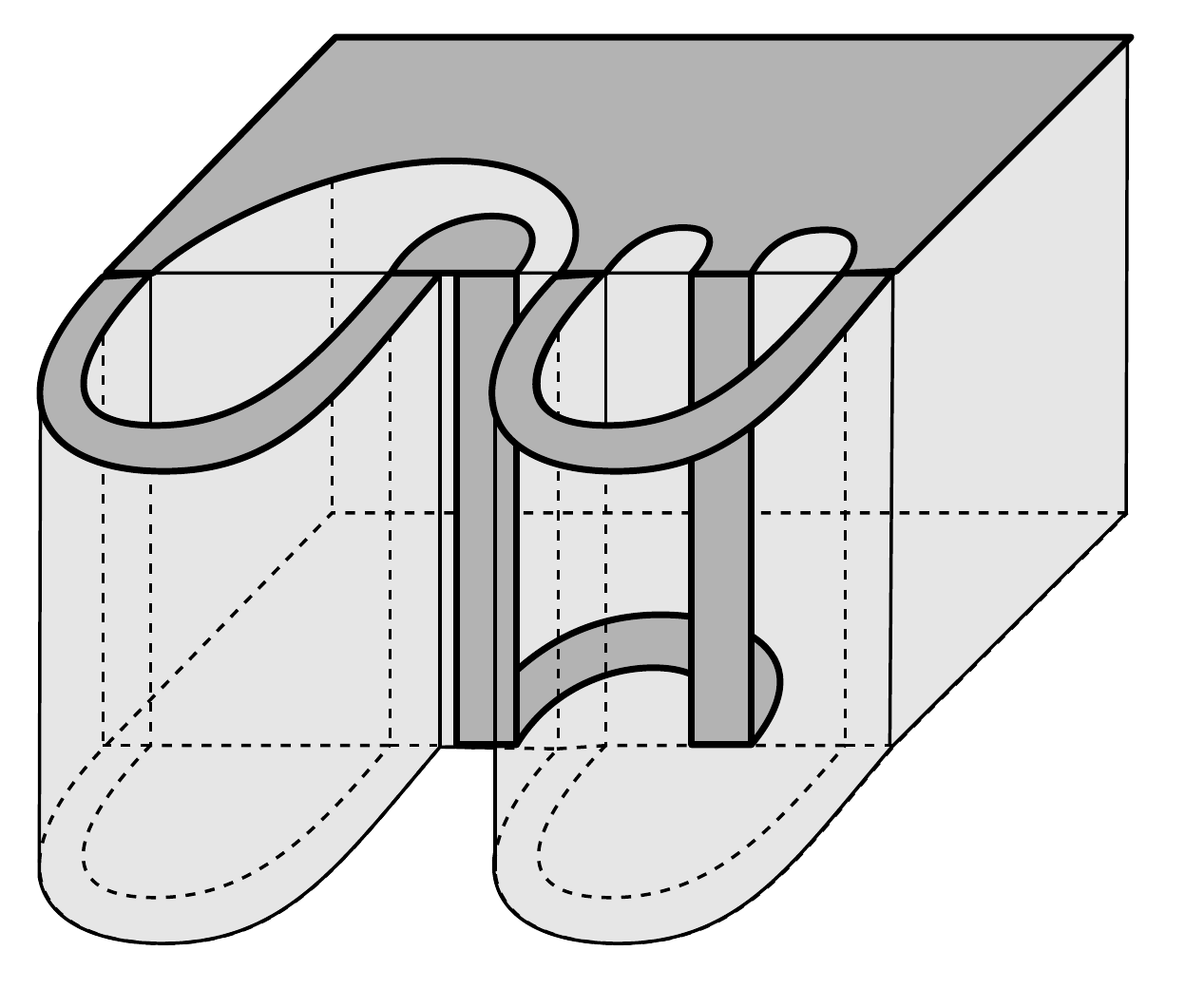}\\
\caption{A diffeomorphism from $\Sigma_{1,1}\times I$ to $\Sigma_{0,3}\times I$.}\label{fig:handleslide}
\end{center}
\end{figure}

\begin{prop}
The map $\mathrm{H}\colon \mathbb{SL}_1\to \HC_1(g,1)$ is the composition of the homomorphism $C\colon \mathbb{SL}_1\to \HC_1(0,2g+1)$ with the Habegger bijection.
\end{prop} 
\begin{proof}
Habegger's bijection is based on the fact that $\Sigma_{g,1}\times I$ is diffeomorphic to $\Sigma_{0,2g+1}\times I$. Refer to Figure~\ref{fig:handleslide}, which shows a copy of $\Sigma_{1,1}\times I$ on the left, immersed into $\R^3$. The special surface $\Sigma^+=\Sigma_{1,1}\times 1$ is shaded more darkly. To get a diffeomorphism to $\Sigma_{0,3}\times I$, take the left-most attaching rectangle of the right-hand handle, slide it up onto the top of the block, move it to the right and then back down onto the side of the block facing the viewer. The surface $\Sigma^+$ is moved to the position indicated on the right in Figure~\ref{fig:handleslide}. Habegger maps string links in homology balls to homology cylinders by reparameterizing the boundary as indicated.

To make this more explicit, we redraw $\Sigma_{0,3}\times I$ as in Figure~\ref{fig:Sigma-plus}(A), keeping track of the surface $\Sigma^+$, which we isotope to a more convenient position in Figure~\ref{fig:Sigma-plus}(B), (C) and (D). Now Habegger's construction, applied to a string link, can be seen as plugging in an arbitrary string link $T$ (with trivial linking matrix) into the trivial string link represented in Figure~\ref{fig:Sigma-plus}(D). This is depicted in Figure~\ref{fig:Habbeger-map}(A). Finally, an isotopy of the surface $\Sigma^+$ leads us to the picture in Figure~\ref{fig:Habbeger-map}(B). Comparing this with Figure~\ref{fig:stringclosure}, it is clear that the map $\mathrm{H}$ coincides with Habegger's construction. 
\end{proof}

\begin{figure}

\subfloat[]{\includegraphics[width=.24\linewidth]{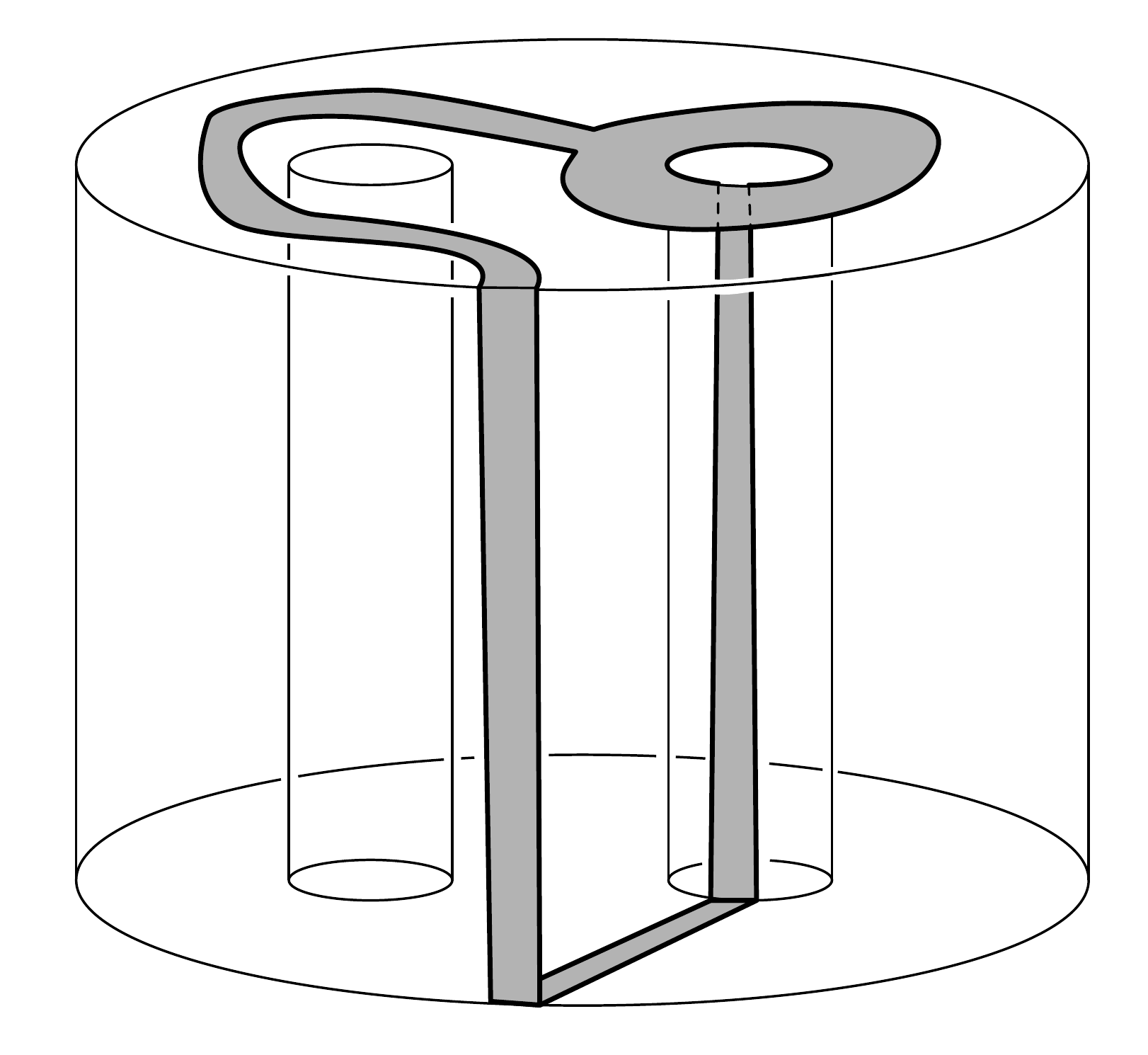}}
\subfloat[]{ \includegraphics[width=.24\linewidth]{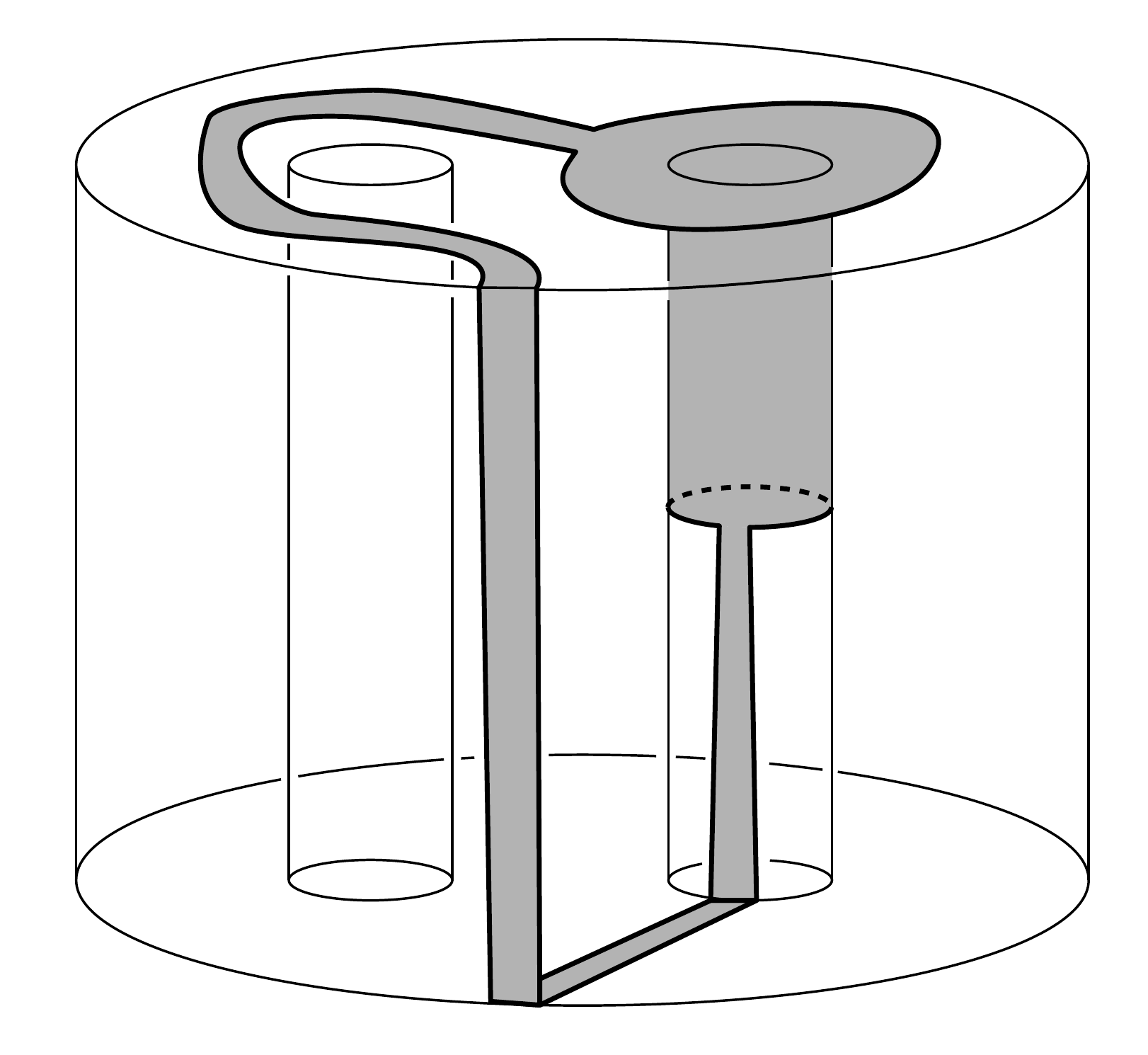}}
\subfloat[]{ \includegraphics[width=.24\linewidth]{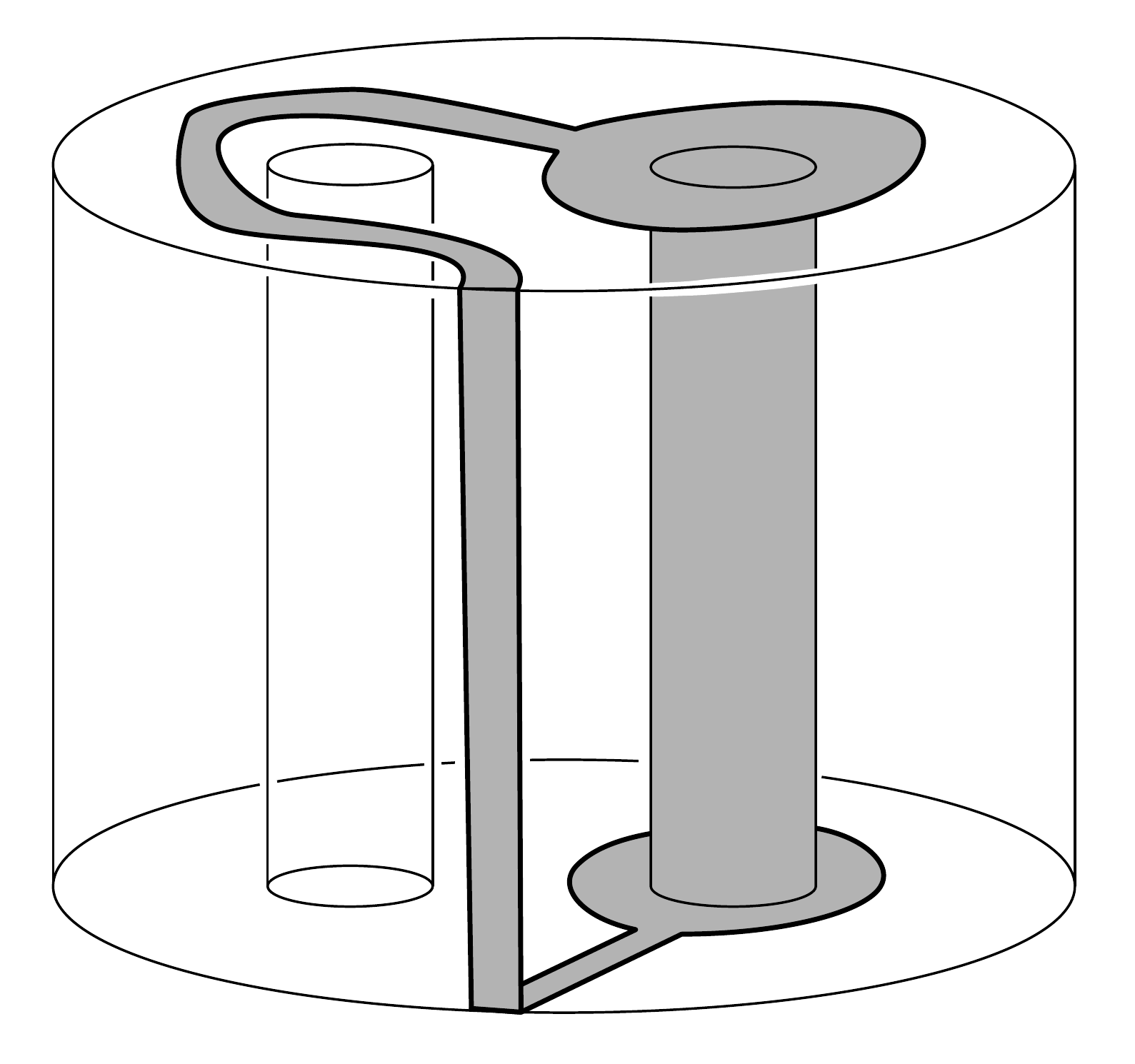}}
\subfloat[]{ \includegraphics[width=.24\linewidth]{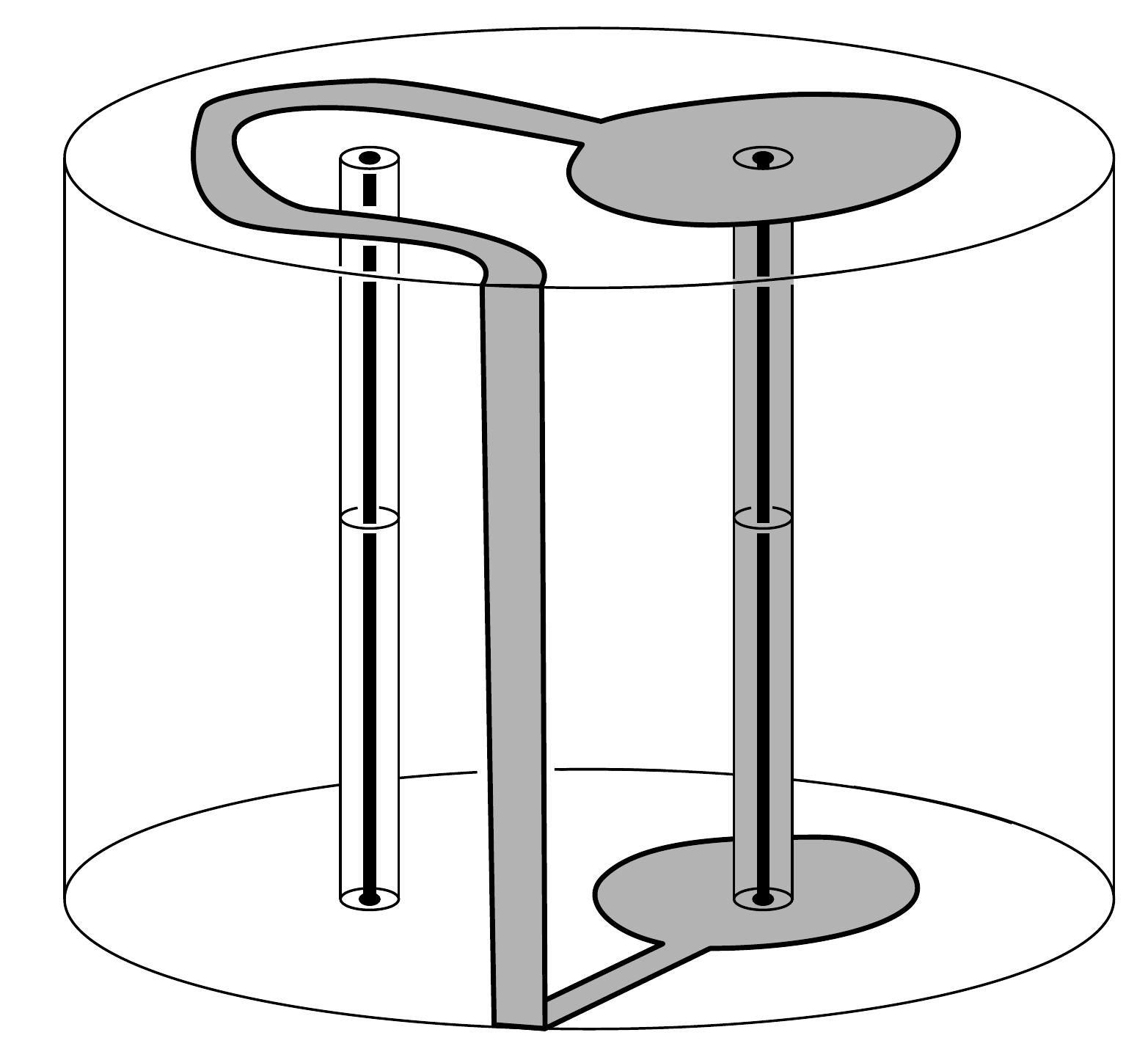}}
\caption{Isotoping $\Sigma^+$.}\label{fig:Sigma-plus}
\end{figure}
\begin{figure}
\subfloat[]{\includegraphics[width=2in]{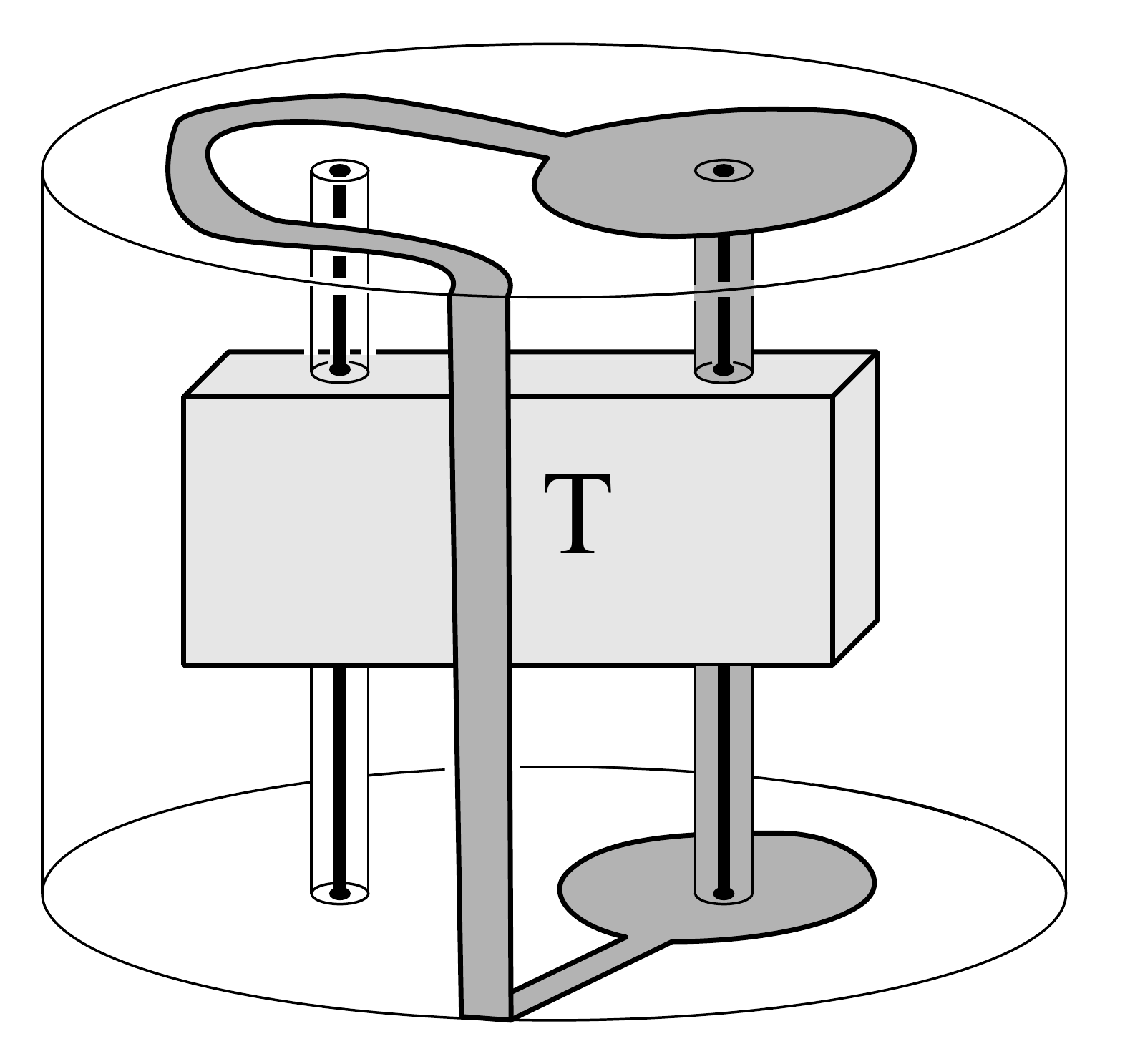}}
\subfloat[]{\includegraphics[width=2in]{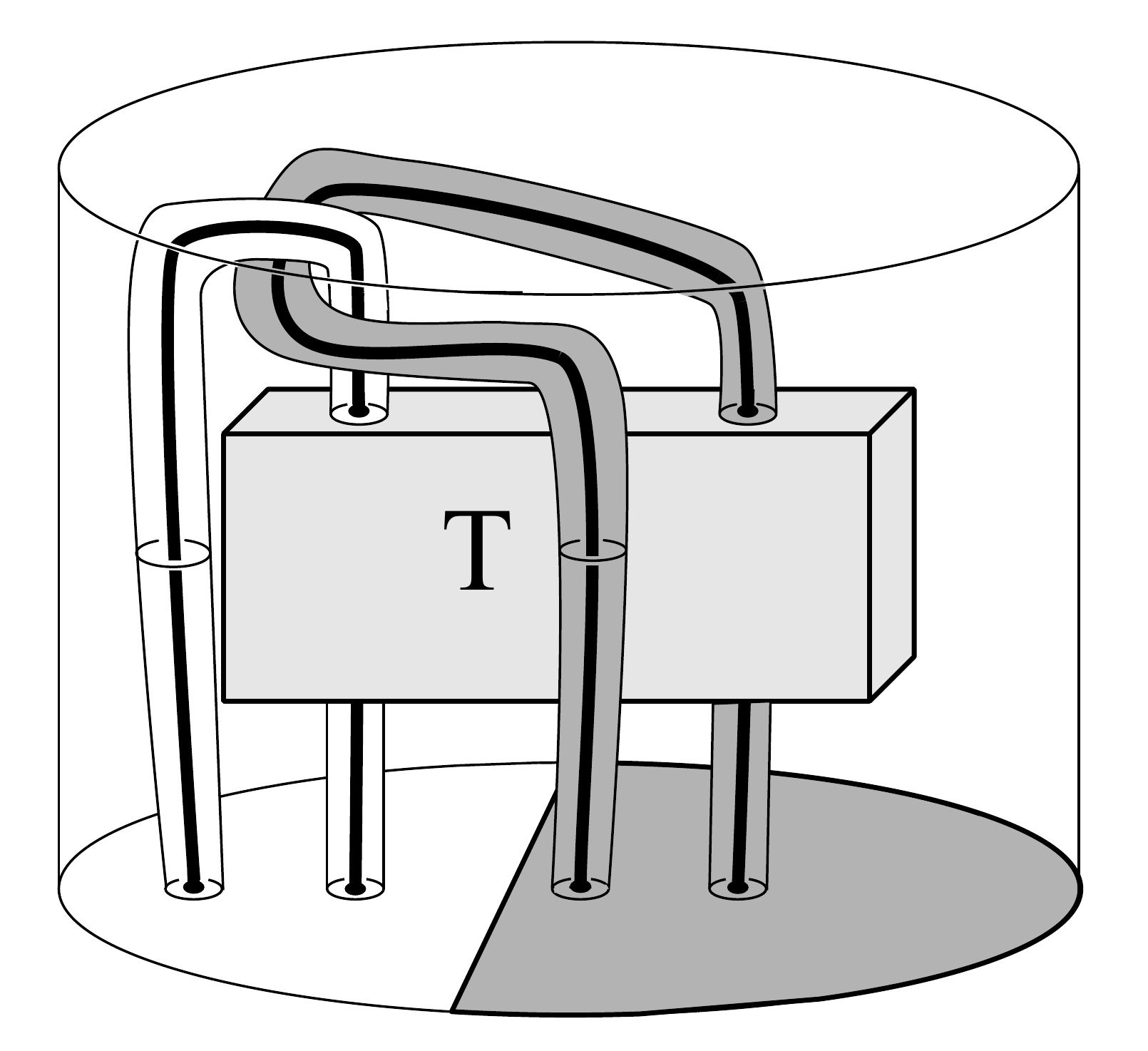}}
\caption{ Habegger's map.}\label{fig:Habbeger-map}
\end{figure}

Returning to the main discussion,
it is not immediately obvious whether $\mathrm{H}(\mathbb {SW}_n)\subset \mathbb Y_n$. However we can define the following intermediate filtration on $\mathbb{SL}$: Denote by $\SY_n$ the set of concordance classes of string links which differ from the trivial string link by a sequence of concordances and surgeries along \emph{simple} claspers of order $n$, where a
simple clasper is defined to so that its leaves bound disjointly embedded meridianal disks to the string link components.

\begin{prop}\label{prop:inclusions}
For all $n\geq 1$, we have inclusions $\SY_n\subset\mathbb {SW}_n$ and $\mathrm{H}(\SY_n)\subset \mathbb Y_n$.
\end{prop}
\begin{proof}
A sequence of surgeries on simple claspers of order $n$ is the same as a $3$-dimensional capped grope cobordism of class $n+1$ by \cite[Thm. 23]{CST} (using \cite{CT1}), which gives rise to a Whitney tower concordance of order $n$, as described in  \cite[sec. 3.5]{CST} (using \cite{S1}). Thus $\SY_n\subset\mathbb {SW}_n$. On the other hand $\mathrm{H}(\SY_n)\subset \mathbb Y_n$ follows because both filtrations are defined in terms of surgery on claspers (and clasper surgery obviously commutes with the map $\mathrm{H}$).
\end{proof}

Recall from Theorem~\ref{thm:nilpotent} that the standard closure operation gives an isomorphism 
$$
\sSW_n\cong\sW_n
$$ 
between the graded groups $\sSW_n=\mathbb {SW}_n/\mathbb{SW}_{n+1}$ associated to the string link Artin filtration the group of classical links bounding order $n$ Whitney towers modulo Whitney tower concordance of order $n+1$.  The rest of this section uses algebraic results from \cite{CST1} to compare the Johnson and $Y$-filtrations (working with $\sW_n$ for ease of reference to \cite{CST1}).

Recall Habiro's surjective realization map $\theta_n:\cT_n\twoheadrightarrow{\sf Y}_n$  $(n>1)$ from Section~\ref{sec:homology-cylinders} which can just as easily be defined using only simple clasper surgeries, thinking of the identity element in the group of homology cylinders as the image of the trivial string link under $\mathrm{H}$.
  There is a surjective realization map $\cT_n\twoheadrightarrow\sSY_n$ defined in an analogous way by embedding simple claspers in the complement of the identity string link, where $\sSY_n$ is the quotient of $\SY_n$ by the equivalence relation generated by order $n+1$ simple clasper surgery and concordance. There is also a realization surjection $R_n:\cT_n\twoheadrightarrow\sW_n$
defined in \cite{CST1} based on Cochran's bing-doubling construction \cite{Co}.

The following proposition combines Levine's Theorem~\ref{thm:Levine} with the basic set-up of the Whitney tower theory
surveyed in \cite{CST0}, and serves as a starting point for this section's results.

\begin{prop}\label{prop:comm}
The following diagram is commutative for all $n>1$:

%\centerline{$\xymatrix{
%\cT_n\ar@{->>}[dr]\ar@/^1.5pc/[drr]\ar@/_1.5pc/[ddr]\ar@/^1pc/[ddrr]_(.8){\eta_n}&&\\
%&{\sf Y}^s_{n}\ar@{->>}[d]\ar@{->>}[r]^{K}&{\mathsf Y}_{n}\ar@{->}[d]\\
%&{\sf W}_{n}\ar@{->}[r]&{\sf D}_n
%}$}
\centerline{$\xymatrix{
&&&\\
&&{\sf W}_n\ar[dr]&\\
\cT_n\ar@/^5pc/[rrr]^{\eta_n}\ar@{->>}[rru]\ar@{->>}[r]\ar@{->>}[rrd]&\sSY_n\ar@{->>}[ur]\ar@{->>}[dr]^{\mathrm{H}_n}&&{\sf D}_n\\
&&{\sf Y}_n\ar[ur]&
}$}
Hence all maps are rational isomorphisms.
\end{prop}
\begin{proof}
The bottom row factorization of $\eta_n$ is Theorem~\ref{thm:Levine} from Section~\ref{sec:homology-cylinders}.
The map $\sW_n\to{\sf D}_n$ is given by order $n$ Milnor invariants $\mu_n$, and the factorization $\eta_n=\mu_n\circ R_n$
is fundamental to the theory developed in \cite{CST0,CST1,CST2,CST3}. 
Commutativity of the bottom left triangle follows because the realization maps both are defined by embedding claspers. 
Commutativity of the triangle containing $\sSY_n\to\sW_n$ corresponds to the fact that the conversion from simple claspers to capped gropes to Whitney towers preserves tree type \cite{CST,S1}. Now, because $\cT_n\to  \sSY_n$ is surjective, diagram chasing is enough to show the right diamond commutes, and that the two maps emanating from $\sSY_n$ are epimorphisms. 
\end{proof}
\begin{rem}
As mentioned in the introduction, in a future paper we will give a geometric argument that the map $\sSY_n\to \sSW_n\cong \sW_n$ is an isomorphism. This gives another argument that Habiro's surgery map induces a map $\widetilde{\cT}_n\to \sY_n$. However, we have presented the direct clasper arguments in this paper because they appear to give another fruitful line of approach to the theory.

% One is that they give another approach to our theory. Secondly, the details of the argument that $\SY_n=\SW_n$ are roughly as complicated as the clasper arguments, so for the purposes of this paper the two are roughly equivalent.
\end{rem}

To make progress understanding the analogous setting with $\Z$-coefficients, we appeal to the resolution of the Levine Conjecture in \cite{CST3} and the resulting implications described in \cite{CST1}.

As mentioned in Section~\ref{sec:homology-cylinders} above, in \cite{L2,L3} Levine studied a map $\eta'_n\colon\cT_n\to\sD'_n$, where $\sD_n'$ is the kernel of the quasi-Lie bracket map $\sL_1'\otimes\sL'_{n+1}\to\sL'_{n+2}$, which is defined by the same formula as $\eta_n$.
In \cite{L3} it is shown that ${\sf D}'_{2k}<{\sf D}_{2k}$,
% with quotient ${\sf D}_{2k}/{\sf D}'_{2k}\cong \Z_2\otimes{\sf L}_{k+1}$,
and in \cite{CST3} it is shown that in all even orders the maps $\eta'_{2k}=\eta_{2k}:\cT_{2k}\to{\sf D}'_{2k}$ are isomorphisms. This implies
the following corollary to Proposition~\ref{prop:comm}.
\begin{cor}
For all $k\geq 1$, the following diagram is commutative and all maps are $\Z$-module isomorphisms:

\centerline{$\xymatrix{
&&&\\
&&{\sf W}_{2k}\ar@{>->>}[dr]&\\
\cT_{2k}\ar@/^5pc/[rrr]^{\eta'}\ar@{>->>}[rru]\ar@{>->>}[r]\ar@{>->>}[rrd]&\sSY_{2k}\ar@{>->>}[ur]
\ar@{>->>}[dr]^{H_{2k}}&&{\sf D}'_{2k}\\
&&{\sf Y}_{2k}\ar@{>->>}[ur]&
}$}
%\centerline{$\xymatrix{
%\cT_{2k}\ar@{>->>}[dr]\ar@/^1.5pc/[drr]\ar@/_1.5pc/[ddr]\ar@/^1pc/[ddrr]_(.8){\eta_{2k}}&&\\
%&{\sf Y}^s_{{2k}}\ar@{>->>}[d]\ar@{>->>}[r]^{K}&{\mathsf Y}_{{2k}}\ar@{>->>}[d]\\
%&{\sf W}_{{2k}}\ar@{>->>}[r]&{\sf D}'_{2k}
%}$}
$\hfill\square$
\end{cor}

In odd orders the situation depends on the order's parity modulo $4$: Recall from \cite[Sec.6]{CST1} that 
$\widetilde{\sD}_{2k-1}$ is defined to be the quotient of 
$\sD'_{2k-1}$ by the image under $\eta_{2k-1}':\cT_{2k-1}\to\sD'_{2k-1}$ of the framing relations (Definition~\ref{def:reduced-tree-group} above) in $\widetilde{\cT}_{2k-1}$, and the induced map $\widetilde{\eta}_{2k-1}:\widetilde{\cT}_{2k-1}\to\widetilde{\sD}_{2k-1}$ is an isomorphism for all $k\geq 1$. 

In odd orders congruent to $3$ mod $4$, we get an analogous diagram of abelian groups and isomorphisms.

\begin{cor}
For all $k\geq 1$, the following diagram is commutative and all maps are $\Z$-module isomorphisms:

\centerline{$\xymatrix{
&&&\\
&&{\sf W}_{4k-1}\ar@{>->>}[dr]&\\
\widetilde{\cT}_{4k-1}\ar@/^5pc/[rrr]^{\widetilde{\eta}_{4k-1}}\ar@{>->>}[rru]\ar@{>->>}[r]\ar@{>->>}[rrd]&\sSY_{4k-1}\ar@{>->>}[ur]\ar@{>->>}[dr]^{H_{4k-1}}&&\widetilde{\sf D}_{4k-1}\\
&&{\sf Y}_{4k-1}\ar@{>->>}[ur]&
}$} 

\end{cor}
\begin{proof}

Consider the commutative diagram in Proposition~\ref{prop:comm}. The realization maps emanating from ${\cT}_{4k-1}$ are well-defined on the quotient $\widetilde{\cT}_{4k-1}$. This follows from Theorem~1.5 of \cite{CST1} for the upper map, Proposition~\ref{prop:well-defined} for the bottom map, and by an analogous argument applied to simple claspers on string links modulo concordance for the middle map. We also know, combining Theorems 5.1 and 6.5(ii) of \cite{CST1},  that there is an exact sequence
$$0\to \z\otimes\sL_{2k+1}\to \widetilde{\cT}_{4k-1}\to \sD_{4k-1}\to 0.$$ 
It follows that $\z\otimes\sL_{2k+1}$ surjects onto the kernels of both ${\sf Y}_{4k-1}\to {\sf D}_{4k-1}$ and ${\sf W}_{4k-1}\to {\sf D}_{4k-1}$. Furthermore, by Theorem~\ref{thm:improved-filtration} and  \cite[Thm. 6.1]{CST1}, the kernels of both $\sY_{4k-1}\to\sD_{4k-1}$ and $\sW_{4k-1}\to \sD_{4k-1}$ are isomorphic to $\z\otimes\sL_{2k+1}$. It follows by the $5$-lemma that the epimorphisms $\widetilde{\cT}_{4k-1}\to \sY_{4k-1}$ and $\widetilde{\cT}_{4k-1}\to \sW_{4k-1}$ are actually isomorphisms. Furthermore, the fact that $\eta_{4k-1}$ can be lifted to an isomorphism to $\widetilde{\sD}_{4k-1}$ implies the maps $\sY_{4k-1},\sW_{4k-1}\to {\sD}_{4k-1}$ can be lifted to maps, hence isomorphisms, to $\widetilde{\sD}_{4k-1}$ as well.
\end{proof}

\subsection{Proof of Theorem~\ref{thm:H}}\label{subsec:proof-thm-H}
\begin{proof}
The last two corollaries prove the first sentence of Theorem~\ref{thm:H}. The statement for $4k+1$ follows from Proposition~\ref{prop:comm} and the fact that the kernel of the map $\widetilde{\cT}_{4k+1}\to \sD_{4k+1}$ is the $2$-torsion group $\z\otimes\sL'_{2k+2}$ , which follows from Theorems 1.13 and 5.1 of \cite{CST1}.
\end{proof}

\subsection{Higher-order Arf invariants for homology cylinders}\label{subsec:higher-order-Arf-for-homology-cylinders}
Finally we consider the conjectural set-up in the order $4k+1$ case, where the higher-order Arf invariants live (for framed filtrations). Recall from \cite[Sec. 6.2]{CST1}  that the nontriviality of the higher-order Arf invariants for links is equivalent to the existence of a map $\widetilde{\mu}_{4k+1}\colon {\sf W}_{4k+1}\to\widetilde{\sf D}_{4k+1}$ such that the composition $\widetilde{\cT}_{4k+1}\to {\sf W}_{4k+1}\to\widetilde{\sf D}_{4k+1}$ is $\widetilde{\eta}_{4k+1}$. By a similar argument, Conjecture~\ref{conj:a-maps} is true if and only if there is a map $\widetilde{\mu}^{\sf Y}_{4k+1}\colon{\sf Y}_{4k+1}\to\widetilde{\sf D}_{4k+1}$ so that $\widetilde{\theta}_{4k+1}\circ \widetilde{\mu}^{\sf Y}_{4k+1}=\widetilde{\eta}_{4k+1}$. So we get a commutative diagram of conjectured isomorphisms for $k>0$:

\centerline{$\xymatrix{
&&&\\
&&{\sf W}_{4k+1}\ar@{>-->>}[dr]_{\widetilde{\mu}_{4k+1}}&\\
\widetilde{\cT}_{4k+1}\ar@/^5pc/[rrr]^{\widetilde{\eta}}\ar@{>->>}[rru]\ar@{>->>}[r]\ar@{>->>}[rrd]&\sSY_{4k+1}\ar@{>-->>}[rr]_{\widetilde{\mu}^s_{4k+1}}\ar@{>->>}[ur]\ar@{->>}[dr]^{H_{4k+1}}&&\widetilde{\sf D}_{4k+1}\\
&&{\sf Y}_{4k+1}\ar@{>-->>}_{\widetilde{\mu}^{\sf Y}_{4k+1}}[ur]&
}$}

Assuming our announced isomorphism $\sSY_{n}\cong\sW_n$, the existence of  $\widetilde{\mu}^s_{4k+1}$ is equivalent to the existence of  $\widetilde{\mu}_{4k+1}$. On the other hand, if $\widetilde{\mu}^{\sf Y}_{4k+1}$ exists, then so does $\widetilde{\mu}_{4k+1}$.  So if Conjecture~\ref{conj:a-maps} is true, it would imply the existence of higher-order Arf invariants for links.

When $k=0$, we almost have a diagram of isomorphisms, but there is an extra $\z$ in ${\sf Y}_1$ coming from the Rochlin invariant. That is, we have ${\sf Y}_1\cong {\sf W}_1\oplus\z$ and $\widetilde{\mathcal{T}}_1\cong\widetilde{\sf D}_1\cong {\sf W}_1\cong {\sSY}_1$.

 %=============================================================== 

%\Addresses

\end{document}